\documentclass[]{amsart}

\usepackage{amsmath,amsthm,amssymb}
\usepackage{mathrsfs}
\usepackage[all]{xy}
\usepackage{tikz}
\usepackage{diagbox}
\usepackage{hyperref}

\newtheorem{theorem}{Theorem}[section]

\newtheorem{lemma}[theorem]{Lemma}
\newtheorem{corollary}[theorem]{Corollary}

\theoremstyle{definition}
\newtheorem{definition}[theorem]{Definition}
\newtheorem{example}[theorem]{Example}
\newtheorem{remark}[theorem]{Remark}

\newcommand{\holim}{\operatornamewithlimits{holim}}

\title{Triple delooping for multiplicative hyperoperads}

\author{Florian De Leger}
\address{Mathematical Institute of the Academy \\ \v{Z}itn\'a 25, 115~67 Prague 1, Czech Republic}
\email{de-leger@math.cas.cz}

\author{Maro\v{s} Grego}
\address{Faculty of Mathematics and Physics, Charles University \\
Sokolovsk\'a 49/83, 186~75 Prague 8, Czech Republic}
\email{maros@grego.site}

\thanks{The first author is supported by RVO:67985840 and Praemium Academiae of Martin Markl}
	
\begin{document}
	
\maketitle

\begin{abstract}
	Using techniques developed in \cite{batanindeleger}, we extend the Turchin/Dwyer-Hess double delooping result to further iterations of the Baez-Dolan plus construction. For $0 \leq m \leq n$, we introduce a notion of $(m,n)$-bimodules which extends the notions of bimodules and infinitesimal bimodules over the terminal non-symmetric operad. We show that a double delooping always exists for these bimodules. For the triple iteration of the Baez-Dolan construction starting from the initial $1$-coloured operad, we provide a further reduceness condition to have a third delooping.
\end{abstract}
	
\tableofcontents

\section{Introduction}

The goal of this paper is to extend the Turchin/Dwyer-Hess double delooping result to further iterations of the Baez-Dolan plus construction.

The Turchin/Dwyer-Hess theorem \cite{turchin,dwyerhess} concerns multiplicative (non-symmetric) operads. The notion of non-symmetric operad will be recalled in Example \ref{noppolynomialmonad}. A non-symmetric operad $\mathcal{O}$ is called \emph{multiplicative} when it is equipped with an operad map $Ass \to \mathcal{O}$, where $Ass$ is the terminal non-symmetric operad. Such a map endows the collection $(\mathcal{O}_n)_{n \geq 0}$ with a structure of cosimplicial object \cite{turchin}, which we will write $\mathcal{O}^\bullet$. The theorem states that if a multiplicative operad $\mathcal{O}$ is \emph{reduced}, that is $\mathcal{O}_0$ and $\mathcal{O}_1$ are contractible, then there is a double delooping
\begin{equation}\label{theoremtdh}
	\Omega^2 \mathrm{Map}_{\mathrm{NOp}} (Ass, u^*(\mathcal{O})) \sim \mathrm{Tot}(\mathcal{O}^\bullet),
\end{equation}
where $\mathrm{Map}$ is the homotopy mapping space, taken in the category $\mathrm{NOp}$ of non-symmetric operads, $u^*$ is the forgetful functor from multiplicative to non-symmetric operads and $\mathrm{Tot}$ is the homotopy totalization. This result is remarkable especially because of an earlier result of Sinha \cite{sinha1} which states that the space of \emph{long knots modulo immersions} \cite{dwyerhess} is equivalent to the totalization of the Kontsevich operad. The double delooping \ref{theoremtdh} has been extended to general deloopings in higher dimensions in \cite{ducoulombier,ducoulombierturchin}. Our goal is also to extend \ref{theoremtdh} but in another direction.

Non-symmetric operads appear when we iterate the Baez-Dolan plus construction, as we will now explain. This construction was first introduced in \cite{baezdolan} in order to define weak $n$-categories. It is a construction which associates to a (symmetric coloured) operad $P$ a new operad $P^+$ for \emph{operads over $P$}. More explicitly, if $P$ is an $S$-coloured operad, $P^+$ is the operad whose algebras are $S$-coloured operads equipped with an operad map to $P$. For example, if $I$ is the initial $1$-coloured operad, a $1$-coloured operad is equipped with an operad map to $I$ if and only if it is a monoid, so $I^+$ is the operad for monoids $Ass$ (as a symmetric operad this time). Iterating this construction, one gets the operad $I^{++}$ for non-symmetric operads. This construction can of course be iterated infinitely many times. The question that naturally arises and that we will try to answer in this paper is if there are delooping results analogous to \ref{theoremtdh} for the next iterations of the plus construction.

We will work with polynomial monads, which are equivalent to symmetric coloured operads with freely acting symmetric groups \cite{kock}. We will recall the notion in Section \ref{sectionpreliminaries}, as well as the description of the plus construction for polynomial monads from \cite{KJBM}. We iterate this construction from the identity monad on the category of sets, which corresponds to the initial $1$-coloured operad. This gives us a sequence of polynomial monads which we call the \emph{opetopic sequence}. Our main tool in order to get delooping results such as \ref{theoremtdh} is the extension of some homotopy theory results from small categories to polynomial monads. Therefore we will recall an important notion of \cite{batanindeleger}, namely the notion of \emph{homotopically cofinal} map of polynomial monads.

In Section \ref{sectionbimodules}, we introduce the notion of \emph{$(m,n)$-bimodule}, for $0 \leq m \leq n$. Let us motivate this notion now. The proofs of the double delooping \ref{theoremtdh} presented in \cite{turchin,dwyerhess}, as well as in \cite{batanindeleger}, all proceed in two steps. Indeed, we have the deloopings
\begin{equation}\label{firsttdhdelooping}
	\Omega \mathrm{Map}_{\mathrm{NOp}}(Ass, u^*(\mathcal{O})) \sim \mathrm{Map}_\mathrm{Bimod}(Ass,v^*(\mathcal{O}))
\end{equation}
if $\mathcal{O}_1$ is contractible, and
\begin{equation}\label{secondtdhdelooping}
	\Omega \mathrm{Map}_{\mathrm{Bimod}}(Ass, v^*(\mathcal{O})) \sim \mathrm{Map}_\mathrm{IBimod}(Ass,w^*(\mathcal{O}))
\end{equation}
if $\mathcal{O}_0$ is contractible, where $\mathrm{Bimod}$ and $\mathrm{IBimod}$ are the category of bimodules over $Ass$ and infinitesimal bimodules over $Ass$ respectively, $Ass$ is seen as both a bimodule and infinitesimal bimodule over itself and $v^*$ and $w^*$ are the appropriate forgetful functors. Since infinitesimal bimodules over $Ass$ are known to be equivalent to cosimplicial objects \cite{turchincosimplicial}, we do indeed get the double delooping \ref{theoremtdh}. Our notion of $(m,n)$-bimodule is an extension of bimodule over $Ass$ and infinitesimal bimodule over $Ass$ to further iterations of the plus construction.

In Section \ref{sectioncofinality}, we construct a map of polynomial monads whose algebras involve these $(m,n)$-bimodules. We prove in Theorem \ref{thmcofinalcospan} that this map is homotopically cofinal, which extends a result of \cite{batanindeleger}.

In Section \ref{sectiondeloopingtheorems}, we investigate a general delooping for $(m,n)$-bimodules. More precisely, we try to exhibit reduceness conditions for a delooping
\[
	\Omega \mathrm{Map}_{\mathrm{Bimod}_{m,n}} (\zeta,u^* (\mathcal{O})) \sim \mathrm{Map}_{\mathrm{Bimod}_{m-1,n}} (\zeta, v^* (\mathcal{O})),
\]
where $\zeta$ is our notation for the terminal object in the category $\mathrm{Bimod}_{m,n}$ of $(m,n)$-bimodules. We prove in Theorem \ref{theoremtdhgeneral} that in the cases $m=n$ and $m=n-1$, the reduceness conditions are analogous to the ones for the Turchin/Dwyer-Hess theorem. Theorem \ref{theoremtdhgeneral} gives us in particular the deloopings \ref{firsttdhdelooping} and \ref{secondtdhdelooping}. We further investigate the third iteration of the plus construction, starting from the identity monad. We start by recalling the definition of the category $\Omega_p$, the version of the dendroidal category for planar trees \cite{moerdijktoen}. We prove that $\Omega_p$ has the same role as the simplex category $\Delta$ has in the Turchin/Dwyer-Hess theorem. We define a notion of \emph{functor equipped with retractions} for a covariant presheaf over $\Omega_p$, which we use to exhibit a reduceness condition for a third delooping. The triple delooping we get in Corollary \ref{corollarytripledelooping} is the analogue of the double delooping \ref{theoremtdh} for the next iteration of the plus construction. Finally, we apply our triple delooping to a non-trivial example, namely the desymmetrisation of the Kontsevich operad.

In a future work, we would like to investigate the geometric meaning of our triple delooping and try to find out if there is an analogue to Sinha's result in this case. We would also like to explore the connections between our direction of delooping results and the one from \cite{ducoulombierturchin}.

\subsection*{Acknowledgement}

Both authors are deeply grateful to Michael Batanin who suggested this project, for his guidance and the many illuminating discussions during our weekly meetings.

\section{Preliminaries}\label{sectionpreliminaries}

\subsection{Polynomial monads}

Recall \cite{gambinokock} that a polynomial monad $T$ is a cartesian monad whose underlying functor is given by the composite $t_! p_* s^*$ for some diagram in $\mathrm{Set}$, the category of sets, of the form
\begin{equation}\label{diagrampolynomial}
	\xymatrix{
		I & E \ar[l]_s \ar[r]^p & B \ar[r]^t & I
	}
\end{equation}
where $p^{-1}(b)$ is finite for all $b \in B$. The elements of the sets $I$, $B$ and $E$ will be called \emph{colours}, \emph{operations} and \emph{marked operations} respectively. The maps $s$, $p$ and $t$ will be called \emph{source map}, \emph{middle map} and \emph{target map} respectively. The diagram \ref{diagrampolynomial} will be called the \emph{polynomial} for the monad.

An algebra of such polynomial monad in a symmetric monoidal category $(\mathcal{E},\otimes,e)$ is given by a collection $(A_i)_{i \in I}$ together with, for all $b \in B$, structure maps
\[
	m_b: \bigotimes_{e \in p^{-1}(b)} A_{s(e)} \to A_{t(b)}
\]
satisfying associativity and unitality axioms.

\begin{example}\label{monpolynomialmonad}
	The free monoid monad $\mathbf{Mon}$ is a polynomial monad \cite[Example 2.6]{batanindeleger} given by
	\[
		\xymatrix{
			1 & Ltr^* \ar[l] \ar[r] & Ltr \ar[r] & 1,
		}
	\]
	where $Ltr$ is the set of (isomorphism classes of) linear trees, $Ltr^*$ is the set of linear trees with one vertex marked, the middle map forgets the marking. Multiplication is given by insertion of a linear tree inside a vertex.
\end{example}

\begin{example}\label{noppolynomialmonad}
	Recall that a non-symmetric operad $A$ in a symmetric monoidal category $(\mathcal{E},\otimes,e)$ in given by a collection of objects $A_n \in \mathcal{E}$ for $n \geq 0$ together with maps
	\[
		A_k \otimes A_{n_1} \otimes \ldots \otimes A_{n_k} \to A_{n_1+\ldots+n_k}.
	\]
	and a map $e \to A_1$ satisfying associativity and unitality axioms. The polynomial $\mathbf{NOp}$ for non-symmetric operads \cite[Example 2.7]{batanindeleger} is given by
	\[
		\xymatrix{
			\mathbb{N} & Ptr^* \ar[l] \ar[r] & Ptr \ar[r] & \mathbb{N},
		}
	\]
	where $Ptr$ is the set of (isomorphism classes of) planar trees, $Ptr^*$ is the set of planar trees with one vertex marked, the source map returns the number of edges directly above the marked vertex, the middle map forgets the marking, the target map returns the number of leaves. Multiplication is given by insertion of a planar tree inside a vertex.
\end{example}

\subsection{Baez-Dolan plus construction}\label{subsectionbdplusconstruction}

The Baez-Dolan plus construction was first introduced in \cite{baezdolan}. Let us recall this construction for polynomial monads from \cite{KJBM}. Assume $T$ is a polynomial monad whose underlying polynomial is given by diagram \ref{diagrampolynomial}. We construct a new polynomial monad $T^+$:
\[
\xymatrix{
	B & tr(T)^* \ar[l] \ar[r] & tr(T) \ar[r] & B,
}
\]
where $tr(T)$ is the set of $T$-trees, that is trees whose vertices are decorated with elements of $B$ and edges are decorated with elements of $I$, satisfying the coherence condition that if a vertex decorated with $b \in B$, its outcoming edge is decorated with $t(b)$ and its incoming edges are decorated with $s(e)$, for $e \in p^{-1}(b)$:
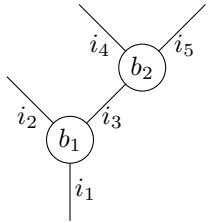
\[
\begin{tikzpicture}[scale=1.2]
\draw (0,-.9) -- (0,0) -- (-.7,.7);
\draw (0,0) -- (1.5,1.5);
\draw (.8,.8) -- (.1,1.5);

\draw[fill=white] (0,0) circle (.26) node{$\tiny{b_1}$};
\draw[fill=white] (.8,.8) circle (.26) node{$\tiny{b_2}$};
\draw (-.05,-.55) node[right]{$\tiny{i_1}$};
\draw (-.25,.25) node[left]{$\tiny{i_2}$};
\draw (.25,.25) node[right]{$\tiny{i_3}$};
\draw (.55,1.05) node[left]{$\tiny{i_4}$};
\draw (1.05,1.05) node[right]{$\tiny{i_5}$};
\end{tikzpicture}
\]
$tr(T)^*$ is the set of $T$-trees with one vertex marked. The source map returns the element which decorates the marked vertex, the middle map forgets the marking, the target map is given by composition of all the elements which decorate the vertices. Multiplication is given by insertion of a tree inside a vertex.

\subsection{Opetopic sequence}

For $n \geq 0$, let us define the polynomial monad $\mathbf{Id}^{+n}$ by induction. $\mathbf{Id}^{+0} = \mathbf{Id}$, that is the identity monad on $\mathrm{Set}$. For $n > 0$, $\mathbf{Id}^{+n} = \left(\mathbf{Id}^{+(n-1)}\right)^+$. One can check that $\mathbf{Id}^{+1} = \mathbf{Mon}$ and $\mathbf{Id}^{+2} = \mathbf{NOp}$, the polynomial monads defined in Example \ref{monpolynomialmonad} and Example \ref{noppolynomialmonad} respectively. We will write the underlying polynomial of $\mathbf{Id}^{+n}$ as follows:
\begin{equation}\label{polynomialopetopicsequence}
\xymatrix{
	I_n & E_n \ar[l]_{s_n} \ar[r]^{p_n} & B_n \ar[r]^{t_n} & I_n.
}
\end{equation}
Note that for $n > 0$, $I_n = B_{n-1}$. To avoid too heavy notations, we will simply write $s$, $t$ and $p$ instead of $s_n$, $t_n$ and $p_n$ respectively when $n \geq 0$ is clear from the context.

\subsection{Homotopically cofinal maps of polynomial monads}

Since a polynomial monad $T$ is in particular cartesian, it makes sense to talk about lax morphisms of categorical $T$-algebras.

\begin{definition}
	Let $f: S \to T$ be a map of polynomial monads and $A$ a categorical $T$-algebra. An \emph{internal $S$-algebra} in $A$ is a lax morphism of $T$-algebras $1 \to f^*(A)$, where $1$ is the terminal $T$-algebra and $f^*$ is the restriction functor.
\end{definition}

We have the following theorem \cite{batanin}:

\begin{theorem}
	Let $f: S \to T$ be a map of polynomial monads. The $2$-functor
	\[
		Int_S: \mathrm{Alg}_T(\mathrm{Cat}) \to \mathrm{Cat},
	\]
	which sends a categorical $T$-algebra $A$ to the category of internal $S$-algebras in $A$, is representable. We will write $T^S$ the representing object and call it \emph{classifier induced by $f$}.
\end{theorem}

\begin{remark}
	Let $f: S \to T$ be a map of polynomial monads given by $\phi: I \to J$ on colours. It was proved in \cite{batanin} that the classifier induced by $f$ can be computed as a truncated simplicial $T$-algebra
	\begin{equation}\label{formulaclassifier}
		\xymatrix{
			F_T(\phi_!(1)) \ar[r] & F_T(\phi_!(S 1)) \ar@<-1ex>[l] \ar@<1ex>[l] & F_T(\phi_!(S^2 1)), \ar[l] \ar@<-1ex>[l] \ar@<1ex>[l]
		}
	\end{equation}
	where $1$ is the terminal $I$-collection, $F_T$ is the free $T$-algebra functor and $\phi_!$ is the left adjoint of the restriction $\phi^*$, given by coproduct over fibres of $\phi$.
\end{remark}

\begin{remark}
	The classifier $T^S$ being a categorical $T$-algebra, it has an underlying collection of categories. We will consider classifiers as categories, implying implicitly that we are talking about an arbitrary category of the underlying collection. Any statement made about a classifier, seen as a category, will imply that it is true for any category of the underlying collection.
\end{remark}

\begin{example}\label{examplecategorystructureonbn}
	For a polynomial monad $T$ given by diagram \ref{diagrampolynomial}, we describe the classifier induced by the identity on the polynomial monad $T^+$. The set of objects is the set of $T$-trees $tr(T)$. Morphisms are contractions of edges and multiplication of the elements of the set $B$ of operations of the polynomial monad $T$ accordingly, or insertion of unary vertices decorated with the unit. In particular, this gives us a category structure on $B_n$.
\end{example}

We now recall an important notion of \cite{batanindeleger}:

\begin{definition}
	A map of polynomial monads $f: S \to T$ is \emph{homotopically cofinal} if $N(T^S)$ is contractible.
\end{definition}

\section{Bimodules in the opetopic sequence}\label{sectionbimodules}

\subsection{$\downarrow$-construction}

\begin{definition}
	For $n \geq 0$, we call \emph{tree with white vertices} a pair $(b,W)$, where $b \in B_n$ and $W$ is a subset of the set of vertices of $b$. We call \emph{white vertices} the elements of $W$. The other vertices of $b$ will be called \emph{black vertices}.
\end{definition}

\begin{definition}\label{constructiontree}
	For $n > 0$, we associate to a tree with white vertices $(b \in B_n,W)$, a tree with white vertices $(b^\downarrow \in B_{n-1},W^\downarrow)$ as follows. Let $b_0$ be the maximal subtree of $b$ containing the root edge and only black vertices. We take $b^\downarrow = t(b_0) \in I_n = B_{n-1}$. Note that the vertices of $b^\downarrow$ correspond to the leaves of $b_0$ which are also edges of $b$. We define $W^\downarrow$ as the set of vertices of $b^\downarrow$ which correspond to an internal edge in $b$.
\end{definition}
	
\begin{example}
	Let $(b,W)$ be the planar tree in the following picture:
	
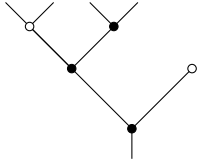
\[
	\begin{tikzpicture}[scale=.8]
	\draw (0,-.5) -- (0,0) -- (-2.1,2.1);
	\draw (-1,1) -- (-1.7,1.7);
	\draw (-1.7,1.7) -- (-1.3,2.1);
	\draw (-.3,1.7) -- (-.7,2.1);
	\draw (-1,1) -- (-.3,1.7);
	\draw (-.3,1.7) -- (.1,2.1);
	\draw (0,0) -- (1,1);
	\draw[fill] (0,0) circle (2pt);
	\draw[fill] (-1,1) circle (2pt);
	\draw[fill=white] (-1.7,1.7) circle (2pt);
	\draw[fill] (-.3,1.7) circle (2pt);
	\draw[fill=white] (1,1) circle (2pt);
	\end{tikzpicture}
	\]
	Then $b_0$ will be the planar tree on the left of the following picture and $(b^\downarrow,W^\downarrow)$ will be the linear tree on the right:
	
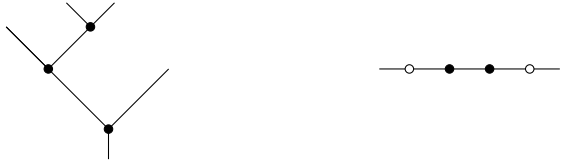
\[
		\begin{tikzpicture}[scale=.8]
			\draw (0,-.5) -- (0,0) -- (-1.7,1.7);
			\draw (-1,1) -- (-1.7,1.7);
			\draw (-.3,1.7) -- (-.7,2.1);
			\draw (-1,1) -- (-.3,1.7);
			\draw (-.3,1.7) -- (.1,2.1);
			\draw (0,0) -- (1,1);
			\draw[fill] (0,0) circle (2pt);
			\draw[fill] (-1,1) circle (2pt);
			\draw[fill] (-.3,1.7) circle (2pt);
			
			\begin{scope}[shift={(6,1)}]
			\draw (-1.5,0) -- (1.5,0);
			\draw[fill=white] (-1,0) circle (2pt);
			\draw[fill] (-.333,0) circle (2pt);
			\draw[fill] (.333,0) circle (2pt);
			\draw[fill=white] (1,0) circle (2pt);
			\end{scope}
		\end{tikzpicture}
	\]
	Indeed, $b^\downarrow$ has four vertices which correspond to leaves of $b_0$. The first and last vertices of $b^\downarrow$ are white because they correspond to leaves in $b_0$ which are internal edges in $b$.
\end{example}

The following lemma will be useful later:

\begin{lemma}\label{lemmainsertingtrunk}
	Let $(b,W)$ be a tree with white vertices and $(b^\downarrow,W^\downarrow)$ be the tree obtained by applying the construction of Definition \ref{constructiontree}. Let $\tilde{b}^\downarrow$ be a tree obtained from $b^\downarrow$ by adding unary vertices and $\tilde{W}^\downarrow$ be the set $W^\downarrow$ plus the unary vertices which have been added. Then there is a tree with white vertices $(\tilde{b},\tilde{W})$ such that $(\tilde{b}^\downarrow,\tilde{W}^\downarrow)$ is the tree obtained from it by applying the construction of Definition \ref{constructiontree}.
\end{lemma}

\begin{proof}
	Recall from Example \ref{examplecategorystructureonbn} that $B_n$ has a category structure where morphisms are contractions of edges and insertion of unary vertices. Let us assume that the maximal subtree $b_0$ of $b$ containing the root and only black vertices has been contracted to a corolla. This can be done without loss of generality because the tree obtained by applying the construction of Definition \ref{constructiontree} remains the same. Assume $\tilde{b}^\downarrow$ is obtained from $b^\downarrow$ by adding a unary vertex decorated with $\eta(c) \in B_{n-2}$, where $c \in I_{n-2}$ is the decoration of the edge where the unary vertex is added. Then we take $\tilde{b}$ as the tree obtained from $b_0$ by adding a trunk above the root vertex. The vertex of the trunk is decorated with $\nu(c) \in B_{n-1}$, where $\nu(c)$ is the free living edge decorated with $c$. The edge of the trunk is decorated with $t \nu(c) = \eta(c)$. The tree $b^\downarrow$ which decorates the root vertex of $b$ is replaced by $\tilde{b}^\downarrow$.
	
	For example, in the following picture, the tree on the left is $\tilde{b}$, and the inserted trunk is dotted. The tree on the right is $\tilde{b}^\downarrow$, and the inserted unary vertex is dotted. It has four vertices, including the added unary vertex, since the root vertex of $\tilde{b}$ has four edges above it, including the added trunk:
	
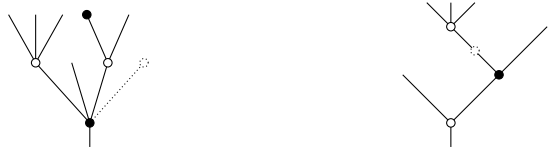
\[
	\begin{tikzpicture}[scale=.8]
	\draw (0,-.4) -- (0,0) -- (-.9,1);
	\draw (-.3,1) -- (0,0) -- (.3,1);
	\draw[densely dotted] (0,0) -- (.9,1);
	\draw (-1.35,1.8) -- (-.9,1) -- (-.9,1.8);
	\draw (-.9,1) -- (-.45,1.8);
	\draw (-.05,1.8) -- (.3,1) -- (.65,1.8);
	\draw[fill] (0,0) circle (2pt);
	\draw[fill=white] (-.9,1) circle (2pt);
	\draw[fill=white] (.3,1) circle (2pt);
	\draw[fill=white,densely dotted] (.9,1) circle (2pt);
	\draw[fill] (-.05,1.8) circle (2pt);
	\begin{scope}[shift={(6,0)}]
	\draw (0,-.4) -- (0,0) -- (-.8,.8);
	\draw (0,0) -- (1.6,1.6);
	\draw (.8,.8) -- (-.4,2);
	\draw (0,2) -- (0,1.6) -- (.4,2);
	\draw[fill=white] (0,0) circle (2pt);
	\draw[fill] (.8,.8) circle (2pt);
	\draw[fill=white,densely dotted] (.4,1.2) circle (2pt);
	\draw[fill=white] (0,1.6) circle (2pt);		
	\end{scope}
	\end{tikzpicture}
	\]
	
	This can of course be done for multiple unary vertices added.
\end{proof}

\subsection{$m$-dimensional sets of vertices}

For $0 \leq i \leq n$ and a tree with white vertices $(b \in B_n,W)$, we will write $(b^{\downarrow i},W^{\downarrow i})$ for the tree with white vertices obtained by iterating $i$ times the construction of Definition \ref{constructiontree}.

\begin{definition}\label{definitionmdimensionalsetofvertices}
	Let $0 \leq m \leq n$ and $(b,W)$ a tree with white vertices. We say that $W$ is \emph{$m$-dimensional} if
	\begin{itemize}
		\item for $0 \leq i < n-m$, $W^{\downarrow i}$ does not contain any pairs of vertices one above the other,
		\item $W^{\downarrow (n-m)}$ is the set of vertices of $b^{\downarrow (n-m)}$.
	\end{itemize}
\end{definition}

\begin{remark}\label{remarkallmdimensional}
	It is immediate from Definition \ref{definitionmdimensionalsetofvertices} that if $(b,W)$ is $m$-dimensional, then $(b^{\downarrow i},W^{\downarrow i})$ is $m$-dimensional for all $0 \leq i \leq n-m$.
\end{remark}

\begin{remark}\label{remarkbijections}
	Let $(b,W)$ be a tree with white vertices such that $W$ is $m$-dimensional. Then all the sets $W^{\downarrow i}$ for $0 \leq i \leq n-m$ are in bijection with each other. Indeed, the construction of Definition \ref{constructiontree} induces a surjection from $W^{\downarrow i}$ to $W^{\downarrow (i+1)}$ for $0 \leq i < n-m$. The first condition of Definition \ref{definitionmdimensionalsetofvertices} implies that these surjections are also injective.
\end{remark}

\begin{remark}\label{casezeroandn}
	Let $n > 0$ and $(b \in B_n,W)$ be a tree with white vertices. $W$ is $n$-dimensional if it is the set of vertices of $b$. It is $0$-dimensional if it is a singleton. It is $n-1$-dimensional if each path from the root to a leaf in $b$ contains exactly one vertex of $W$. Indeed, the first condition of Definition \ref{definitionmdimensionalsetofvertices} ensures that any path contains at most one vertex, while the second one ensures that any path contains at least one vertex.
\end{remark}

\subsection{$(m,n)$-bimodules}

\begin{definition}\label{operationsbimodmn}
	For $0 \leq m \leq n$, let $B_{m,n}$ be the set of trees with white vertices $(b,W)$ where $b \in B_n$, $W$ is $m$-dimensional and the tree does not contain any pairs of adjacent black vertices or any unary black vertices.
\end{definition}

\begin{definition}
	Let $\mathbf{Bimod}_{m,n}$ be the polynomial monad represented by
	\[
		\xymatrix{
			I_n & E_{m,n} \ar[l]_s \ar[r]^p & B_{m,n} \ar[r]^t & I_n,
		}
	\]
	where $E_{m,n}$ is the set of elements of $B_{m,n}$ with a marked white vertex. The source, target and middle map are defined as in Subsection \ref{subsectionbdplusconstruction}. Multiplication is given by inserting a tree inside a white vertex, then contracting all edges between two black vertices and removing all unary black vertices.
\end{definition}

\begin{definition}\label{definitionmnbimodule}
	For $0 \leq m \leq n$, an \emph{$(m,n)$-bimodule} in a symmetric monoidal category $(\mathcal{E},\otimes,e)$ is an algebra of $\mathbf{Bimod}_{m,n}$ in $\mathcal{E}$.
\end{definition}

For an $I_n$-collection $A$ in $\mathcal{E}$, $b \in B_n$ and $V \subset p^{-1}(b)$, we will write
\[
	A_{s(V)} = \bigotimes_{v \in V} A_{s(v)}.
\]

\begin{remark}
	Explicitly, an \emph{$(m,n)$-bimodule} in a symmetric monoidal category $(\mathcal{E},\otimes,e)$ is given by
	\begin{itemize}
		\item a collection of objects $(A_i)_{i \in I_n}$ in $\mathcal{E}$,
		\item for all $(b,V) \in B_{m,n}$, a map
		\[
			\mu_{b,V}: A_{s(V)} \to A_{t(b)},
		\]
	\end{itemize}
	satisfying associativity and unitality axioms.
\end{remark}

\begin{remark}
	According to Remark \ref{casezeroandn}, $B_{n,n} = B_n$, so the polynomial monad $\mathbf{Bimod}_{n,n}$ is just $\mathbf{Id}^{+n}$. Also according to Remark \ref{casezeroandn}, the elements of $B_{0,n}$ are trees with exactly one white vertex. We deduce that $E_{0,n} = B_{0,n}$ and the middle map for the polynomial monad $\mathbf{Bimod}_{0,n}$ is the identity, so this polynomial monad is actually a small category.
\end{remark}

\begin{remark}\label{examplelowercases}
	According to Remark \ref{casezeroandn}, the sets $B_{2,2}$, $B_{1,2}$ and $B_{0,2}$ are the sets of planar trees, planar trees with white vertices, called \emph{beads} in \cite{turchin}, lying on the same horizontal line and planar trees with one distinguished bead, respectively. One recognises the polynomial monad $\mathbf{Bimod}_{m,2}$, when $m=2$, $m=1$ and $m=0$, as the monad for non-symmetric operads, bimodules over the terminal non-symmetric operad $Ass$ and infinitesimal bimodules over $Ass$ described in \cite{batanindeleger}, respectively.
\end{remark}

\subsection{Universal classifier for $(m,n)$-bimodules}\label{subsectionuniversalclassifierbimodmn}

Let us now describe the classifier induced by the identity on the polynomial monad $\mathbf{Bimod}_{m,n}$. The set of objects is $B_{m,n}$. The morphisms can be given by nested trees, that is trees $(b,W) \in B_{m,n}$, where each vertex $v \in W$ has itself a tree inside it, called nest. The following picture is an example of such nested tree:
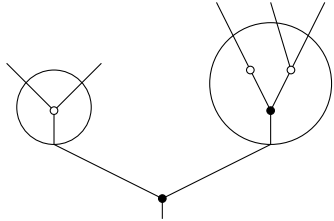
\[
\begin{tikzpicture}[scale=.9]
\draw (.4,-.1) -- (.4,.2) -- (-1.2,1);
\draw (.4,.2) -- (2,1);
\draw (-1.2,1) -- (-1.2,1.5) -- (-1.9,2.2);
\draw (-1.2,1.5) -- (-.5,2.2);
\draw (2,1) -- (2,1.5) -- (1.7,2.1) -- (1.2,3.1);
\draw (2,1.5) -- (2.8,3.1);
\draw (2.3,2.1) -- (2,3.1);

\draw (2,1.9) circle (.9);
\draw (-1.2,1.55) circle (.55);
\draw (-1.75,1.55);
\draw[fill=white] (-1.2,1.5) circle (1.6pt);
\draw[fill] (.4,.2) circle (1.6pt);
\draw[fill] (2,1.5) circle (1.6pt);
\draw[fill=white] (2.3,2.1) circle (1.6pt);
\draw[fill=white] (1.7,2.1) circle (1.6pt);
\draw (2.9,1.9);
\end{tikzpicture}
\]
The source of the nested tree is obtained by inserting the nests into each corresponding vertex and then contracting the edges connecting black vertices if necessary. The target is obtained by forgetting the nests. For example, the nested tree of the previous picture represents the following morphism:
\[
\begin{tikzpicture}[scale=.7]
\draw (0,-.1) -- (0,.2) -- (-1.3,1) -- (-1.9,1.6);
\draw (-1.3,1) -- (-.7,1.6);
\draw (0,1.6) -- (0,.2) -- (1.3,1) -- (1.9,1.6);
\draw (1.3,1) -- (.7,1.6);
\draw[fill] (0,.2) circle (2pt);
\draw[fill=white] (-1.3,1) circle (2pt);
\draw[fill=white] (0,1) circle (2pt);
\draw[fill=white] (1.3,1) circle (2pt);

\draw (3.2,.8) node{$\longrightarrow$};

\begin{scope}[shift={(6,0)}]
\draw (0,-.1) -- (0,.2) -- (-.9,1) -- (-1.5,1.6);
\draw (-.9,1) -- (-.3,1.6);
\draw (0,.2) -- (.9,1) -- (1.5,1.6);
\draw (.9,1.6) -- (.9,1) -- (.3,1.6);
\draw[fill] (0,.2) circle (2pt);
\draw[fill=white] (-.9,1) circle (2pt);
\draw[fill=white] (.9,1) circle (2pt);
\end{scope}
\end{tikzpicture}
\]

\subsection{Bimodules over $(m,n)$-bimodules}

\begin{lemma}\label{lemmamminusonedimensional}
	Let $(b,W) \in B_{m,n}$ with an $m-1$-dimensional $V \subset W$. Then the complement of $V$ in $W$ is the canonical union of two sets $V_-$ and $V_+$.
\end{lemma}

\begin{proof}
	If $m=n$, then according to Remark \ref{casezeroandn} $W$ is the set of vertices of $b$ and each path from a leaf to the root in $b$ contains exactly one vertex of $V$. Then we can take $V_-$ and $V_+$ as the set of vertices in $W$ that are below and above a vertex of $V$, respectively.
	
	If $m < n$, let $(b^\downarrow_V,V^\downarrow)$ and $(b^\downarrow_W,W^\downarrow)$ be the pairs obtained by applying the construction of Definition \ref{constructiontree} to $(b,V)$ and $(b,W)$, respectively. According to Remark \ref{remarkallmdimensional}, $V^\downarrow$ and $W^\downarrow$ are $m-1$-dimensional and $m$-dimensional, respectively. It is easy to see from the construction that, since $V \subset W$, $b^\downarrow_W$ can be obtained from $b^\downarrow_V$ by contracting some of its edges, which connect black vertices. So $V^\downarrow$ can be seen as a set of vertices of $b^\downarrow_W$ and it is also $m-1$-dimensional as a set of vertices of this tree. By induction, the complement of $V^\downarrow$ in $W^\downarrow$ is the canonical union of two sets. According to Remark \ref{remarkbijections}, $W$ and $W^\downarrow$, as well as $V$ and $V^\downarrow$, are in bijection. This concludes the proof.
\end{proof}

\begin{definition}\label{defbimodoverbimod}
	For $0 < m \leq n$ and $A$ and $B$ two $(m,n)$-bimodules in $\mathcal{E}$, an \emph{$A-B$-bimodule} $C$ is given by
	\begin{itemize}
		\item a collection of objects $(C_i)_{i \in I_n}$,
		\item for all $(b,W) \in B_{m,n}$ with an $m-1$-dimensional $V \subset W$, a map
		\begin{equation}\label{mappointedbimodule}
			A_{s(V_-)} \otimes C_{s(V)} \otimes B_{s(V_+)} \to C_{t(b)},
		\end{equation}
	\end{itemize}
	where $V_-$ and $V_+$ are given by Lemma \ref{lemmamminusonedimensional}, satisfying associativity and unitality axioms.
\end{definition}

\begin{lemma}\label{lemmabimodulesbimodules}
	For $0 < m \leq n$, let $\zeta$ be the terminal $(m,n)$-bimodule. If $n-2 \leq m$, the category of $(m-1,n)$-bimodules is isomorphic to the category of $\zeta-\zeta$-bimodules.
\end{lemma}

\begin{proof}
	Let $(b,V) \in B_{m-1,n}$. We will construct $(\tilde{b},\tilde{W}) \in B_{m,n}$ such that $\tilde{W}$ contains an $m-1$-dimensional subset which is in bijection with $V$.
	\begin{itemize}
		\item If $m=n$, we take $\tilde{b}=b$ and $\tilde{W}$ as the set of vertices of $b$.
		\item If $m=n-1$, $\tilde{b}$ is the tree obtained from $b$ by adding a unary vertex on each leaf which is not above a vertex of $V$. $\tilde{W}$ is the union of $V$ and all the unary vertices which have been added.
		\item If $m=n-2$, let $(b^\downarrow,V^\downarrow)$ be the pair obtained from $(b,V)$ by applying the construction of Definition \ref{constructiontree}. Let $(\tilde{b}^\downarrow,\tilde{W}^\downarrow)$ be the tree obtained from $(b^\downarrow,V^\downarrow)$ by adding unary vertices as in the case $m=n-1$. We take $(\tilde{b},\tilde{W})$ given by Lemma \ref{lemmainsertingtrunk}.
	\end{itemize}
	It is easy to see that the $\zeta-\zeta$-bimodule structure map induced by $(\tilde{b},\tilde{W})$ corresponds to the $(m-1,n)$-bimodule structure induced by $(b,V)$.
\end{proof}

\subsection{Pointed bimodules over $(m,n)$-bimodules}

\begin{definition}\label{defpointedbimod}
	For $0 < m \leq n$ and $A$ and $B$ two $(m,n)$-bimodules in $\mathcal{E}$, a \emph{pointed $A-B$-bimodule} $C$ is given by
	\begin{itemize}
		\item a collection of objects $(C_i)_{i \in I_n}$,
		\item for all $(b,W) \in B_{m,n}$ and partitions of $W$ into $V_-$, $V_+$ and $V$ such that there is an $m-1$-dimensional subset $U \subset W$ satisfying $U_- \subset V_-$ and $U_+ \subset V_+$, a map
		\[
			A_{s(V_-)} \otimes C_{s(V)} \otimes B_{s(V_+)} \to C_{t(b)},
		\]
	\end{itemize}
	satisfying associativity and unitality axioms.
\end{definition}

\begin{lemma}\label{lemmaalpha}
	For $0 < m \leq n$ and $A$ and $B$ two $(m,n)$-bimodules in $\mathcal{E}$, there is an $A-B$-bimodule $\alpha$ such that the category of pointed $A-B$-bimodules is isomorphic to the comma category of $A-B$-bimodules under $\alpha$.
\end{lemma}

\begin{proof}
	It is obvious that there is a forgetful functor from pointed $A-B$-bimodules to $A-B$-bimodules. Indeed, if one takes $V=U$ in the definition of pointed $A-B$-bimodule, one gets the structure maps for $A-B$-bimodules. Let $\alpha$ be the image of the initial pointed $A-B$-bimodule through this forgetful functor. If $C$ is a pointed $A-B$-bimodule, then it is an $A-B$-bimodule equipped with a map from $\alpha$. Now assume that $C$ is an $A-B$-bimodule equipped with a map from $\alpha$. Note that $\alpha$, being a pointed $A-B$-bimodule, is equipped with maps $A \to \alpha \leftarrow B$ of collections, therefore $C$ is also equipped with such maps. So $C$ has a pointed bimodule structure given by the composite
	\[
		A_{s(V_-)} \otimes C_{s(V)} \otimes B_{s(V_+)} \to A_{s(U_-)} \otimes C_{s(U)} \otimes B_{s(U_+)} \to C_{t(b)},
	\]
	where the first map is given from the maps $A \to C \leftarrow B$ and the second is from the $A-B$-bimodule structure.
\end{proof}

\section{Cofinality}\label{sectioncofinality}

\subsection{Statement of the result}

In this section, we assume $0 < m \leq n$ are fixed. We will closely follow \cite[Section 3]{deleger}. Let $\mathbf{Bimod}_{m,n}^{\bullet+\bullet}$ be the polynomial monad for triples $(A,B,C)$ where $A$ and $B$ are $(m,n)$-bimodules and $C$ is a pointed $A-B$-bimodule. Let $\mathbf{Bimod}_{m,n}^{\bullet \to \bullet \leftarrow \bullet}$ be the polynomial monad for cospans $A \to C \leftarrow B$ of $(m,n)$-bimodules.

\begin{theorem}\label{thmcofinalcospan}
	There is a homotopically cofinal map of polynomial monads
	\begin{equation}\label{cofinalmappolymon}
		f: \mathbf{Bimod}_{m,n}^{\bullet+\bullet} \to \mathbf{Bimod}_{m,n}^{\bullet \to \bullet \leftarrow \bullet}.
	\end{equation}
\end{theorem}

\subsection{Description of the map of polynomial monads}

The polynomial monad $\mathbf{Bimod}_{m,n}^{\bullet+\bullet}$ is given by
\begin{equation}\label{polynomialpointed}
	\xymatrix{
		\{A,B,C\} \times I_n & E_{m,n}^{\bullet+\bullet} \ar[l] \ar[r] & B_{m,n}^{\bullet+\bullet} \ar[r] & \{A,B,C\} \times I_n
	}
\end{equation}
where elements of $B_{m,n}^{\bullet+\bullet}$ are pairs $(b,W) \in B_{m,n}$, equipped with a label in $\{A,B,C\}$ called \emph{target label} and for each white vertex a label in $\{A,B,C\}$ called \emph{source label}, subject to the following restrictions. If the target label is $A$ (resp. $B$), then all the source labels are also $A$ (resp. $B$). If the target label is $C$, there must be an $m-1$-dimensional subset $V \subset W$ such that all the vertices in $V_-$ have label $A$ and all the vertices in $V_+$ have label $B$. $E_{m,n}^{\bullet+\bullet}$ is the set of pairs $(b,W)$ of $B_{m,n}^{\bullet+\bullet}$ with one white vertex marked. Note that there is a projection from \ref{polynomialpointed} to \ref{polynomialopetopicsequence}. The source map is given by the source label of the marked vertex and an element in $I_n$ thanks to this projection. Similarly, the target map is given by the target label and an element in $I_n$ thanks to this projection. The middle map forgets the marking. Multiplication is given by insertion of a tree inside a vertex, if the source and target labels correspond.

The description of the polynomial monad $\mathbf{Bimod}_{m,n}^{\bullet \to \bullet \leftarrow \bullet}$ is completely similar. The difference is that we drop the condition that there must be an $m-1$-dimensional subset $V \subset W$ such that all the vertices in $V_-$ have label $A$ and all the vertices in $V_+$ have label $B$. The map $f$ in \ref{cofinalmappolymon} is given by inclusion of sets.

\subsection{Construction of a smooth functor}

For a functor $F: \mathcal{X} \to \mathcal{Y}$ between categories and $y \in \mathcal{Y}$, we will write $F_y$ for the fibre of $F$ over $y$.

\begin{definition}
	A functor $F : \mathcal{X} \to \mathcal{Y}$ is \emph{smooth} if, for all $y \in \mathcal{Y}$, the canonical functor
	\[
		F_y \to y/F
	\]
	induces a weak equivalence between nerves.
\end{definition}

Let us state the Cisinski lemma \cite[Proposition 5.3.4]{cisinski}:

\begin{lemma}\label{cisinskilemma}
	A functor $F : \mathcal{X} \to \mathcal{Y}$ is smooth if and only if for all maps $f_1 : y_0 \to y_1$ in $\mathcal{Y}$ and objects $x_1$ in $\mathcal{X}$ such that $F(x_1) = y_1$, the nerve of the \emph{lifting category} of $f_1$ over $x_1$, whose objects are arrows $f : x \to x_1$ such that $F(f) = f_1$ and morphisms are commutative triangles
	\[
		\xymatrix{
			x \ar[rr]^g \ar[rd]_f && x' \ar[ld]^{f'} \\
			& x_1
		}
	\]
	with $g$ a morphism in $F_{y_0}$, is contractible.
\end{lemma}

We have a commutative square of polynomial monads
\[
	\xymatrix{
		\mathbf{Bimod}_{m,n}^{\bullet+\bullet} \ar[r]^-{uf} \ar[d]_-f & \mathbf{Bimod}_{m,n} \ar[d]^-{id} \\
		\mathbf{Bimod}_{m,n}^{\bullet \to \bullet \leftarrow \bullet} \ar[r]_-u & \mathbf{Bimod}_{m,n}
	}
\]
where $u$ is given by projection. This square induces a morphism of algebras \cite[Proposition 4.7]{batanindeleger}
\begin{equation}\label{inducedfunctor}
F: (\mathbf{Bimod}_{m,n}^{\bullet \to \bullet \leftarrow \bullet})^{\mathbf{Bimod}_{m,n}^{\bullet+\bullet}} \to u^* \left( (\mathbf{Bimod}_{m,n})^{\mathbf{Bimod}_{m,n}} \right).
\end{equation}
To simplify the notations, we will write $\mathcal{X}$ and $\mathcal{Y}$ for the domain and codomain of $F$ respectively.

\subsection{Proof of cofinality}

\begin{lemma}\label{claimcontractibleliftingcategory}
	For a tree $b \in B_m$, let us consider the category $\mathcal{C}(b)$ whose objects are decorations of the vertices of $b$ by labels in $\{A,B,C\}$, such that there is an $m-1$-dimensional subset for which the vertices below have label $A$ and above have label $B$. The morphisms turn vertices with label $A$ or $B$ to vertices with label $C$. This category has contractible nerve.
\end{lemma}

\begin{proof}
	We proceed by induction on the number of vertices of $b$. If $b$ has no vertices, that is the free living edge, then $\mathcal{C}(b)$ is the terminal category. Now assume $b$ has at least one vertex. Let $\mathcal{A}$ be the full subcategory of $\mathcal{C}(b)$ containing the trees for which the root vertex has label $A$. Let $\mathcal{B}$ be the full subcategory of $\mathcal{C}(b)$ containing the trees for which all the vertices which are not the root have label $B$. The union of $\mathcal{A}$ and $\mathcal{B}$ gives the full category, the intersection is the terminal category, $\mathcal{B}$ consists of a cospan. The category $\mathcal{A}$ is isomorphic to $\prod_{e \in E} \mathcal{C}(b(e))$, where $E$ is the set of edges directly above the root and $b(e)$ is the maximal subtree of $b$ having $e$ as root. So the nerve of $\mathcal{A}$ is contractible by induction.
\end{proof}

Note that an object of $\mathcal{X}$ is an element of the set of operations of the polynomial monad $\mathbf{Bimod}_{m,n}^{\bullet \to \bullet \leftarrow \bullet}$, which in particular gives us a pair $(b,W) \in B_{m,n}$.

\begin{lemma}
	\label{lemmacontractibleliftingcategoryinparticularcase}
	Let $f_1: y_0 \to y_1$ be a map in $\mathcal{Y}$ and $x_1 = (b_1,W_1) \in \mathcal{X}$ such that $F(x_1) = y_1$. If $W_1$ is a singleton, then the nerve of the lifting category $\mathcal{X}(x_1,f_1)$ of $f_1$ over $x_1$ is contractible.
\end{lemma}

\begin{proof}
	Let us prove that the lifting category is trivial or isomorphic to a category of Lemma \ref{claimcontractibleliftingcategory}. First, we will describe the functor $F: \mathcal{X} \to \mathcal{Y}$ more explicitly. As mentioned above, the objects of $\mathcal{X}$ are operations of the polynomial monad $\mathbf{Bimod}_{m,n}^{\bullet \to \bullet \leftarrow \bullet}$. So, they are pairs $(b,W) \in B_{m,n}$, equipped with a target label and, for each white vertex, a source label in $\{A,B,C\}$ such that if the target label is $A$ (resp. $B$), then all the source labels are also $A$ (resp. $B$). The morphisms are given by nested trees, as in Subsection \ref{subsectionuniversalclassifierbimodmn}. It is important to note that  there are morphisms which turn vertices with label $A$ or $B$ to vertices with label $C$. The set of objects of $\mathcal{Y}$ is just $B_{m,n}$. Its set of morphisms can again be described in terms of nested trees. The functor $F$ forgets all the labels.
	
	Now let us describe the lifting category. If $W_1$ is a singleton, $x_1$ only depends on the label of the unique element of $W_1$. Let us write $y_0=(b_0,W_0) \in B_{m,n}$. The lifting category has as objects the pairs $(b,W) \in \mathcal{X}$ together with a morphism to $x_1$. We must have $(b_0,W_0) = (b,W)$ as elements of $B_{m,n}$, so the only degree of freedom is in the labels of the white vertices. Since there is a morphism to $x_1$, it means by definition of the category $\mathcal{X}$ that there is an $m-1$-dimensional subset of $W$ for which the vertices below have label $A$ and above have label $B$. The morphisms in the lifting category can only be the morphisms which turn vertices with label $A$ or $B$ to vertices with label $C$. If the label of $W_1$ is $A$ or $B$, then the lifting category is the terminal category. So the only non-trivial case is when the label of $W_1$ is $C$. If $W_0$ is a singleton, then the lifting category consists of a cospan. Otherwise, by definition of $m$-dimensional subset of vertices, there is $b' \in B_m$ obtained by iterating Construction \ref{constructiontree} from $(b_0,W_0)$ such that $p^{-1}(b_0)$ is in bijection with $W_0$. Then the lifting category is isomorphic to $\mathcal{C}(b')$. This concludes the proof.
\end{proof}

\begin{lemma}\label{lemmasmooth}
	The functor $F$ of \ref{inducedfunctor} is smooth.
\end{lemma}

\begin{proof}
	The argument is the same as in \cite[Lemma 3.15]{deleger}. Let $f_1: y_0 \to y_1$ be a map in $\mathcal{Y}$ and $x_1 = (b_1,W_1)$ be an object in $\mathcal{X}$ such that $F(x_1) = y_1$. We want to prove that the lifting category of $f_1$ over $x_1$ has contractible nerve. For a white vertex $v \in W_1$, let $x_1^v$ given the corolla in $B_n$ corresponding to $v$. Let $y_1^v$ be the same corolla but without the labels and $f_1^v: y_0^v \to y_1^v$ be the restriction of $f_1$ for this corolla. The lifting category of $f_1$ over $x_1$ is isomorphic to the product over the white vertices $v$ of the lifting categories of $f_1^v$ over $x_1^v$. It has therefore contractible nerve thanks to Lemma \ref{lemmacontractibleliftingcategoryinparticularcase}. We conclude the proof using Lemma \ref{cisinskilemma}.
\end{proof}

\begin{proof}[Proof of Theorem \ref{thmcofinalcospan}]
	The functor $F$ of \ref{inducedfunctor} is smooth according to Lemma \ref{lemmasmooth}. Its fibres have contractible nerve, since they have a terminal object, which is the object where all the source labels are the same as the target label. Using Quillen Theorem A, we deduce that $F$ induces a weak equivalence between nerve. Again, the nerve of $\mathcal{Y}$ is contractible since this category has a terminal object. So the nerve of $\mathcal{X}$ is contractible, which concludes the proof.
\end{proof}

\section{Delooping theorems}\label{sectiondeloopingtheorems}

\subsection{General delooping for $(m,n)$-bimodules}

For $0 \leq m \leq n$, we will write $\mathrm{Bimod}_{m,n}$ for the category of $(m,n)$-bimodules. We will write $\zeta$ for the terminal object in this category.

\begin{definition}
	For $0 \leq m \leq n$, we will say $X \in \mathrm{Bimod}_{m,n}$ is \emph{multiplicative} if it is equipped with a map $\zeta \to X$.
\end{definition}

For $0 < m \leq n$ fixed, let $\kappa$ be the composite $I_{m-1} \to B_m \to B_n \to I_n$ where the first map picks the free living edge, the second map is obtained by applying the unit $n-m$ times, the last map is the target map $t_n$.

The objective for the rest of this paper is to determine whether, for a multiplicative $(m,n)$-bimodule $X$, there is a fibration sequence
\begin{equation}\label{fibrationsequence}
	\Omega \mathrm{Map}_{\mathrm{Bimod}_{m,n}} (\zeta,u^* X) \to \mathrm{Map}_{\mathrm{Bimod}_{m-1,n}} (\zeta, v^* X) \to \prod_{i \in I^{m-1}} X_{\kappa(i)},
\end{equation}
where $u^*$ and $v^*$ are the appropriate forgetful functors.

\begin{lemma}\label{lemmaleftproper}
	The following categories are left proper:
	\begin{itemize}
		\item the category of $(m,n)$-bimodules,
		\item the category of triples $(A,B,C)$, where $A$ and $B$ are $(m,n)$-bimodules and $C$ is an $A-B$-bimodule,
	\end{itemize}
\end{lemma}

\begin{proof}
	We want to prove that the polynomial monad for $(m,n)$-bimodules is tame \cite[Definition 6.19]{bataninberger}. By definition, we have to prove that the classifier for semi-free coproducts is a coproduct of categories with terminal object. The set of objects for the classifier is the set $B_{m,n}$ of Definition \ref{operationsbimodmn}, where the white vertices are also coloured with $X$ or $K$. The morphisms can be given by nested trees, as it was done in Subsection \ref{subsectionuniversalclassifierbimodmn}. For each vertex of the nested tree, if the vertex is $X$-coloured, the tree inside it can be any tree with all vertices $X$-coloured. If the vertex is $K$-coloured, the tree inside it must be the corolla with the only vertex $K$-coloured. When $m=n$, the local terminal objects are trees in $B_{m,n}$ with white vertices coloured by $X$ and $K$ such that adjacent vertices have different colours, and such that vertices incident to the root or to the leaves are $X$-coloured, as in \cite[Subsection 9.2]{bataninberger}. If $m<n$, let us pick an object of the classifier, that is a tree $(b,W) \in B_{m,n}$ with white vertices coloured by $X$ and $K$. Let $(b^\downarrow,W^\downarrow)$ be the tree obtained by applying the construction of Definition \ref{constructiontree}, with white vertices coloured by $X$ and $K$ as the corresponding vertices of $b$. A morphism from $(b,W)$ corresponds to the contraction of edges and insertion of unary vertices in $b^\downarrow$. If $m = n-1$, the tree $(b,W) \in B_{m,n}$ is again terminal when in $(b^\downarrow,W^\downarrow)$, the $X$-vertices and $K$-vertices alternate. If $m < n-1$, it is terminal when, in $(b^\downarrow,W^\downarrow)$, there are no vertices above an $X$-vertex and all vertices which are below an $X$-vertex are also below a $K$-vertex.
	
	Now let us prove that the polynomial monad for triples $(A,B,C)$, where $A$ and $B$ are $(m,n)$-bimodules and $C$ is an $A-B$-bimodule, is quasi-tame \cite[Definition 4.11]{bwdquasitame}. For $m=n$ and $m=n-1$, this can be done using the fact that the polynomial monad for $(m,n)$-bimodules is tame and applying \cite[Theorem 4.22]{bwdquasitame}. The strategy is completely similar as in the proof of \cite[Proposition 4.26]{bwdquasitame}. If $m < n-1$, the strategy is slightly different. Unfortunately, the subcategory of trees which do not have vertices above an $X$-vertex and such that all vertices which are below an $X$-vertex are also below a $K$-vertex, is not always discrete, because there might be non-trivial morphisms which turn vertices with label $A$ or $B$ to vertices with label $C$. However, it is still final subcategory. Indeed, the inclusion functor of this subcategory has a left adjoint given by the functor which automatically contracts the necessary edges. It also has a contractible nerve, because the only degree of freedom is in the labels of the vertices, so we get a category as in Lemma \ref{claimcontractibleliftingcategory}.
	
	Note that any tame polynomial monad is also quasi-tame \cite[Proposition 4.20]{bwdquasitame}. We get the conclusion from \cite[Theorem 4.17]{bwdquasitame}, which states that the category of algebras over a quasi-tame polynomial monad is left proper.
\end{proof}

\begin{lemma}\label{lemmaquillenequivalences}
	Let $D^0$ be a cofibrant replacement of $\zeta$. There is a Quillen equivalence between the category of pointed $D^0-D^0$-bimodules and the category of pointed $\zeta-\zeta$-bimodules.
\end{lemma}

\begin{proof}
	Let $\mathcal{C}$ be the category of pairs of $(m,n)$-bimodules and $\mathrm{PBimod}_{-,-}: \mathcal{C}^{op} \to \mathrm{CAT}$ the functor which sends a pair $(A,B)$ to the category of pointed $A-B$-bimodules. The Grothendieck construction over this functor is left proper according to Lemma \ref{lemmaleftproper}. The desired Quillen adjunction is induced by the unique map $\tau: (D^0,D^0) \to (\zeta,\zeta)$ in $\mathcal{C}$. To prove that it is indeed a Quillen equivalence, we need to prove that the unique map $0 \to \tau^*(0)$ is a weak equivalence. The initial $D^0-D^0$-bimodule is the left adjoint of the projection from the Grothendieck construction to the base $\mathcal{C}$ applied to the pair $(D^0,D^0)$. The projection is the restriction functor induced by a map of polynomial monads, and the initial $D^0-D^0$-bimodule can be computed as the nerve of the classifier induced by this map. The objects of this classifier are pair $(b,W) \in B_{m,n}$ with white vertices labelled with $A$ or $B$. There must be an $m-1$-dimensional subset $V \subset W$ such that all the vertices in $V_-$ have label $A$ and all the vertices in $V_+$ have label $B$. The morphisms can be given by nested trees, as it was done in Subsection \ref{subsectionuniversalclassifierbimodmn}. The trees inside a nested tree must have all vertices with the same label. If $m=n$ or $m=n-1$, his category is a coproduct of categories with terminal objects. A typical terminal object for $m=n$ has the root vertex labelled with $A$ and all edges above the root vertex is the root of a corolla with label $B$. For $m=n-1$, an object given by $(b,W) \in B_{m,n}$ is terminal when the tree $(b^\downarrow,W^\downarrow)$ has the same description as for the case $m=n$. For $m<n-1$, the classifier is not a coproduct of categories with terminal object. Let us consider the subcategory of objects given by $(b,W) \in B_{m,n}$ such that in $b^\downarrow$, there are no black vertices above a white vertex and a black vertex is below an $A$-vertex if and only if it is also below a $B$-vertex. It is a final subcategory because the inclusion functor has a left adjoint. It is a coproduct of categories with initial object because the only degree of freedom for the morphisms is in the colours of the vertices. The local initial objects are when the number of black vertices is maximal. The initial $\zeta-\zeta$-bimodule is discrete, it is given by the set of connected components of the initial $D^0-D^0$-bimodule. This proves that the unique map $0 \to \tau^*(0)$ is a weak equivalence. Therefore the conditions of \cite[Theorem 3.22]{bwdgrothendieck} are satisfied, which means that $\tau$ indeed induces a Quillen equivalence.
\end{proof}

Let $\alpha$ be the $\zeta-\zeta$-bimodule such that the category of pointed $\zeta-\zeta$-bimodules is equivalent to the comma category of $\zeta-\zeta$-bimodules under $\alpha$, given by Lemma \ref{lemmaalpha}.

\begin{lemma}\label{lemmafibrationsequence}
	For any multiplicative $(m,n)$-bimodule $X$, there is a fibration sequence
	\[
		\Omega \mathrm{Map}_{\mathrm{Bimod}_{m,n}} (\zeta,u^* X) \to \mathrm{Map}_{\mathrm{Bimod}_{\zeta,\zeta}} (\zeta, v^* X) \to \mathrm{Map}_{\mathrm{Bimod}_{\zeta,\zeta}} (\alpha, v^* X),
	\]
	where $\mathrm{Bimod}_{\zeta,\zeta}$ is the category of $\zeta-\zeta$-bimodules.
\end{lemma}

\begin{proof}
	According to Lemma \ref{lemmaleftproper}, $\mathrm{Bimod}_{m,n}$ is left proper. We can therefore apply \cite[Theorem 4.5]{deleger} to get the delooping
	\begin{equation}\label{equationdelooping}
		\Omega \mathrm{Map}_{\mathrm{Bimod}_{m,n}} \left(\zeta,u^* X\right) \to \mathrm{Map}_{S^0/\mathrm{Bimod}_{m,n}} \left(\zeta,h^* X \right),
	\end{equation}
	where $S^0 := D^0 \amalg D^0$.

	The map $f$ of polynomial monads \ref{cofinalmappolymon} induces a Quillen adjunction $f_! \dashv f^*$ between categories of algebras. Note that there is a Quillen adjunction $g_! \dashv g^*$ between $\mathrm{PBimod}_{D^0,D^0}$ and $S^0/\mathrm{Bimod}_{m,n}$, such that $f_!(D^0,D^0,C) = (D^0,D^0,g_!(C))$ and $f^*(D^0,D^0,C) = (D^0,D^0,g^*(C))$. Thanks to Theorem \ref{thmcofinalcospan}, $f$ is homotopically cofinal. According to \cite[Remark 4.8]{deleger}, this means that $f_!$ is \emph{a left cofinal Quillen functor}, that is, it preserves cofibrant replacements of the terminal objects \cite[Definition 4.7]{deleger}. Therefore, $g_!$ is also left cofinal. We deduce by adjunction that there is a weak equivalence
	\begin{equation}\label{equationweakequivalence}
		\mathrm{Map}_{S^0/\mathrm{Bimod}_{m,n}} \left(\zeta,h^* X \right) \to \mathrm{Map}_{\mathrm{PBimod}_{D^0,D^0}} \left(\zeta,g^* h^* X \right).
	\end{equation}
	According to Lemma \ref{lemmaquillenequivalences}, there is a weak equivalence
	\begin{equation}\label{quillenequivalence}
		\mathrm{Map}_{\mathrm{PBimod}_{\zeta,\zeta}} (\zeta, w^*X) \to \mathrm{Map}_{\mathrm{PBimod}_{D^0,D^0}} (\zeta, g^* h^*X).
	\end{equation}
	Using Lemma \ref{lemmaleftproper}, $\mathrm{Bimod}_{\zeta,\zeta}$ is left proper.  According to Lemma \ref{lemmaalpha}, $\mathrm{PBimod}_{\zeta,\zeta}$ is isomorphic to $\alpha / \mathrm{Bimod}_{\zeta,\zeta}$. This means that we can apply \cite[Theorem 4.13]{deleger} and \cite[Proposition 2.7]{rezk} to get the fibration sequence
	\begin{equation}\label{equationfibrationsequence}
	\mathrm{Map}_{\mathrm{PBimod}_{\zeta,\zeta}} \left(\zeta, w^* X \right) \to \mathrm{Map}_{\mathrm{Bimod}_{\zeta,\zeta}} \left(\zeta,v^* X \right) \to \mathrm{Map}_{\mathrm{Bimod}_{\zeta,\zeta}} (\alpha, v^* X).
	\end{equation}
	We get the desired result by combining \ref{equationdelooping}, \ref{equationweakequivalence}, \ref{quillenequivalence} and \ref{equationfibrationsequence}.
\end{proof}

\subsection{The cases $m=n$ and $m=n-1$}

\begin{lemma}\label{lemmamonoids}
	Let $0 < m \leq n$ and a triple $(A,B,C)$ where $A$ and $B$ are $(m,n)$-bimodules in $(\mathcal{E},\otimes,e)$ and $C$ is an $A-B$-bimodule. Let us assume that $m=n$ or $m=n-1$. Then $C$ is pointed if and only if it is equipped with a map $e \to C_{\kappa(i)}$ for all $i \in I_{m-1}$.
\end{lemma}

\begin{proof}
	First let us assume that $C$ is pointed. Let $\lambda$ be the composite $I_{m-1} \to B_m \to B_n$ where the first map picks the free living edge and the second map is obtained by applying the unit $n-m$ times. For $i \in I_{m-1}$, $(\lambda(i),\varnothing)$, where $\varnothing$ is the empty set, is $m$-dimensional. The pointed $A-B$-bimodule map induced by $(\lambda(i),\varnothing)$ is a map $e \to C_{\kappa(i)}$.
	
	Now let us assume that $C$ is equipped with a map $e \to C_{\kappa(i)}$ for all $i \in I_{m-1}$. Let $(b,W) \in B_{m,n}$ and partitions of $W$ into $V_-$, $V_+$ and $V$ such that there is an $m-1$-dimensional subset $U \subset W$ satisfying $U_- \subset V_-$ and $U_+ \subset V_+$. We want to construct a map \ref{mappointedbimodule}.
	
	If $m=n$, according to Remark \ref{casezeroandn}, each path in $b$ from the root to a leaf meets vertices in $V_-$, then at most one vertex in $V$, then vertices in $V_+$. Let $\tilde{b}$ be the tree obtained from $b$ by adding a unary vertex on each edge between a vertex in $V_-$ and a vertex in $V_+$. Let $\tilde{V}$ and $\tilde{W}$ be the sets $V$ and $W$, respectively, plus the set of unary vertices which have been added. Then $(\tilde{b},\tilde{W}) \in B_{m,n}$ and $\tilde{V} \subset \tilde{W}$ is $m-1$-dimensional. The desired map is given by the composite
	\begin{equation}\label{compositepointedbimodule}
		A_{s(V_-)} \otimes C_{s(V)} \otimes B_{s(V_+)} \to A_{s(\tilde{V}_-)} \otimes C_{s(\tilde{V})} \otimes B_{s(\tilde{V}_+)} \to C_{t(b)},
	\end{equation}
	where the first map is using the maps $e \to C_{\kappa(i)}$ for $i \in I_{m-1}$ and the second map is from the $A-B$-bimodule structure.
	
	If $m=n-1$, let $(b^\downarrow,W^\downarrow)$ be the tree obtained by applying the construction of Definition \ref{constructiontree} to $(b,W)$. We can construct $(\tilde{b}^\downarrow,\tilde{W}^\downarrow)$ as in the case $m=n$. Let $(\tilde{b},\tilde{W})$ given by Lemma \ref{lemmainsertingtrunk}. The desired map is given by the composite \ref{compositepointedbimodule} again.
\end{proof}

\begin{theorem}\label{theoremtdhgeneral}
	We do have the fibration sequence \ref{fibrationsequence} if $m=n$ or $m=n-1$.
\end{theorem}

\begin{proof}
	We want to prove that the desired fibration sequence is equivalent to the fibration sequence of Lemma \ref{lemmafibrationsequence}. The first terms of both fibration sequences are the same. According to Lemma \ref{lemmabimodulesbimodules}, the category $\mathrm{Bimod}_{\zeta,\zeta}$ of $\zeta-\zeta$-bimodules is isomorphic to the category $\mathrm{Bimod}_{m-1,n}$, so the second terms are also equivalent. We deduce from Lemma \ref{lemmamonoids} that $\alpha$ is the image of the terminal $I_{m-1}$-collection through the left adjoint of the forgetful functor from $\zeta-\zeta$-bimodules to $I_{m-1}$-collections. So, by an adjunction argument, the third terms are also equivalent.
\end{proof}

\begin{remark}
	Recall from Example \ref{examplelowercases} what are $(m,n)$-bimodules in the case $n=2$. We deduce that Theorem \ref{theoremtdhgeneral} in this case gives us the Turchin/Dwyer-Hess theorem. Interestingly, the fibration sequence \ref{fibrationsequence} does not seem to hold in general for $m<n-1$ without extra assumptions. The rest of this paper will consist in investigating the case $(m,n)=(1,3)$.
\end{remark}

\subsection{Dendroidal category for planar trees $\Omega_p$}

Let $\Omega_p$ be the version of the dendroidal category for planar trees \cite[Definition 2.2.1]{moerdijktoen}. The objects are isomorphism classes of planar trees and the morphisms are generated by:
\begin{itemize}
	\item \emph{inner face maps} of the form $\partial_e: T/e \to T$, where $e$ is an internal edge of $T$ and $T/e$ is the tree obtained from $T$ by contracting $e$:
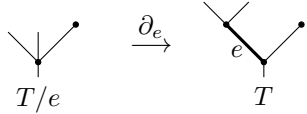
	\[
		\begin{tikzpicture}
			\draw (0,-.2) -- (0,0) -- (-.4,.4);
			\draw (0,.4) -- (0,0) -- (.5,.5);
			\draw[fill] (0,0) circle (1pt);
			\draw[fill] (.5,.5) circle (1pt);
			\draw (0,-.2) node[below]{$T/e$};
			
			\draw (1.5,.2) node{$\longrightarrow$};
			\draw (1.5,.2) node[above]{$\tiny{\partial_e}$};
			
			\begin{scope}[shift={(3,0)}]
			\draw (0,-.2) -- (0,0) -- (-.8,.8);
			\draw[very thick] (0,0) -- (-.5,.5);
			\draw (-.5,.5) -- (-.2,.8);
			\draw (0,0) -- (.5,.5);
			\draw[fill] (0,0) circle (1pt);
			\draw[fill] (-.5,.5) circle (1pt);
			\draw[fill] (.5,.5) circle (1pt);
			\draw (-.15,.35) node[below left]{$e$};
			\draw (0,-.2) node[below]{$T$};
			\end{scope}
		\end{tikzpicture}
	\]
	
	\item \emph{outer face maps} of the form $\partial_v: T/v \to T$, where $v$ is a vertex of $T$, possibly the root, with exactly one inner edge attached to it and $T/v$ is the tree obtained from $T$ by removing the vertex $v$ and all the outer edges incident to it:
	
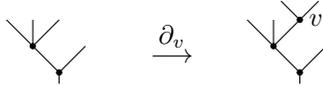
\[
		\begin{tikzpicture}
			\draw (0,-.15) -- (0,0) -- (-.7,.7);
			\draw (-.35,.7) -- (-.35,.35) -- (0,.7);
			\draw (0,0) -- (.35,.35);
			\draw[fill] (0,0) circle (1pt);
			\draw[fill] (-.35,.35) circle (1pt);
			
			\draw (1.5,.2) node{$\longrightarrow$};
			\draw (1.5,.2) node[above]{$\tiny{\partial_v}$};
			
			\begin{scope}[shift={(3.2,0)}]
			\draw (0,-.15) -- (0,0) -- (-.35,.35) -- (.25,.95);
			\draw (0,.7) -- (-.25,.95);
			\draw (-.7,.7) -- (-.35,.35) -- (-.35,.7);
			\draw (0,0) -- (.35,.35);
			\draw[fill] (0,0) circle (1pt);
			\draw[fill] (-.35,.35) circle (1pt);
			\draw[fill] (0,.7) circle (1pt) node[right]{$v$};
			\end{scope}
		\end{tikzpicture}
	\]
	
	\item \emph{degeneracy maps} of the form $\sigma_v: T \to T \backslash v$, where $v$ is a unary vertex of $T$ and $T \backslash v$ is the tree obtained from $T$ by removing the vertex $v$ and merging the two edges incident to it into one:
	
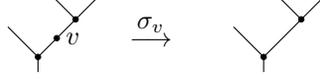
\[
		\begin{tikzpicture}
			\draw (0,-.2) -- (0,0) -- (-.4,.4);
			\draw (0,0) -- (.5,.5) -- (.2,.8);
			\draw (.5,.5) -- (.8,.8);
			
			\draw[fill] (0,0) circle (1pt);
			\draw[fill] (.25,.25) circle (1pt) node[right]{$v$};
			\draw[fill] (.5,.5) circle (1pt);
			
			\draw (1.5,.2) node{$\longrightarrow$};
			\draw (1.5,.2) node[above]{$\tiny{\sigma_v}$};
			
			\begin{scope}[shift={(3,0)}]
			\draw (0,-.2) -- (0,0) -- (-.4,.4);
			\draw (0,0) -- (.5,.5) -- (.2,.8);
			\draw (.5,.5) -- (.8,.8);
			
			\draw[fill] (0,0) circle (1pt);
			\draw[fill] (.5,.5) circle (1pt);
			\end{scope}
		\end{tikzpicture}
	\]
\end{itemize}
These generating morphisms are subject to obvious relations \cite{moerdijktoen}.

\begin{remark}\label{remarkinertactiveomegap}
	Note that $\Omega_p$ has a factorisation system \cite{berger}. Inert maps are generated by the outer face maps. They correspond to full inclusions of trees. Active maps consist of blowing up some vertices, that is inserting new trees inside these vertices. They are generated by the inner face and degeneracy maps.
\end{remark}

\subsection{$(0,3)$-bimodules as covariant presheaves over $\Omega_p$}

We have the following analogue of the well-known equivalence between infinitesimal bimodules over the terminal non-symmetric operad and cosimplicial objects \cite[Lemma 4.2]{turchincosimplicial}.

\begin{lemma}\label{lemmabimodzerothree}
	The category of $(0,3)$-bimodules is isomorphic to the category of covariant presheaves over $\Omega_p$.
\end{lemma}

\begin{proof}
	It comes from the fact that $B_{0,3}$ is in bijection with the set of morphisms of $\Omega_p$. Indeed, according to Remark \ref{casezeroandn}, we deduce that the set $B_{0,3}$ is the set of trees in $B_3$ with exactly one white vertex. The trees whose root vertex is the unique white vertex correspond to active morphisms. The trees without any black vertices above the unique white vertex correspond to inert morphisms.
\end{proof}

\begin{definition}
	We will call \emph{stratum} a connected component of the complement of a planar tree in the plane (we consider that the leaves and the root of the tree go to infinity). For example, in the following picture, the planar tree has four strata:
	
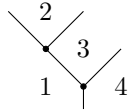
\[
	\begin{tikzpicture}
	\draw (0,-.3) -- (0,0) -- (-1,1);
	\draw (-.5,.5) -- (0,1);
	\draw (0,0) -- (.5,.5);
	\draw (-.5,0) node{$1$};
	\draw (-.5,1) node{$2$};
	\draw (0,.5) node{$3$};
	\draw (.5,0) node{$4$};
	
	\draw[fill] (0,0) circle (1pt);
	\draw[fill] (-.5,.5) circle (1pt);
	\end{tikzpicture}
	\]
	In general, a planar tree with $n$ leaves has $n+1$ strata.
\end{definition}

Let $\alpha: \Omega_p \to \mathcal{E}$ be the functor which sends a planar tree to the coproduct of the unit over all strata of this tree.

\begin{lemma}\label{lemmapbimodzerothree}
	Let $\zeta$ be the terminal $(1,3)$-bimodule. The category of pointed $\zeta-\zeta$-bimodules is isomorphic to the category of covariant presheaves over $\Omega_p$ equipped with a map from $\alpha$ to this presheaf.
\end{lemma}

\begin{proof}
	According to Lemma \ref{lemmabimodulesbimodules} and Lemma \ref{lemmabimodzerothree}, $\zeta-\zeta$-bimodules are equivalent to covariant presheaves over $\Omega_p$. What we need to prove is that $\alpha$ is indeed the initial pointed $\zeta-\zeta$-bimodule of Lemma \ref{lemmaalpha}. As it was pointed out in the proof of Lemma \ref{lemmaquillenequivalences}, the initial pointed $\zeta-\zeta$-bimodule is given by the set of connected components of a classifier. The objects are pairs $(b,W) \in B_{1,3}$. There must be a $0$-dimensional subset $V \subset W$ such that all the vertices in $V_-$ have label $A$ and all the vertices in $V_+$ have label $B$. This means that $b^{\downarrow 2}$ is a linear tree whose vertices have label $A$ or $B$ and all the vertices with label $A$ are on one side and all the vertices with label $B$ are on the other side. This gives labels $A$ and $B$ to the leaves of $t(b)$, where the labels $A$ are to the left and the labels $B$ are to the right. Such labels correspond bijectively to a connected component since it is invariant in each connected component. They also correspond bijectively with the set of strata of $t(b)$. This proves that $\alpha$ is indeed the initial pointed $\zeta-\zeta$-bimodule.
\end{proof}

\subsection{Functors equipped with retractions}

\begin{definition}
	We will say that a morphism in $\Omega_p$ \emph{consists of blowing up a vertex to add a trunk} when it is an inner face map $\partial_e: T/e \to T$, where the vertex directly above $e$ has no inputs:
	
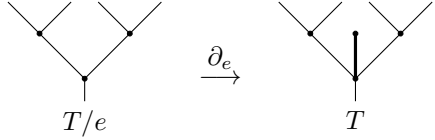
\[
	\begin{tikzpicture}[scale=.6]
	\draw (0,-.5) -- (0,0) -- (-1.7,1.7);
	\draw (-1,1) -- (-.3,1.7);
	\draw (0,0) -- (1.7,1.7);
	\draw (1,1) -- (.3,1.7);
	\draw[fill] (0,0) circle (1.5pt);
	\draw[fill] (-1,1) circle (1.5pt);
	\draw[fill] (1,1) circle (1.5pt);
	\draw (0,-.5) node[below]{$T/e$};
	
	\draw (3,0) node{$\longrightarrow$};
	\draw (3,0) node[above]{$\tiny{\partial_e}$};
	
	\begin{scope}[shift={(6,0)}]
	\draw (0,-.5) -- (0,0) -- (-1.7,1.7);
	\draw (-1,1) -- (-.3,1.7);
	\draw[very thick] (0,1) -- (0,0);
	\draw (0,0) -- (1.7,1.7);
	\draw (1,1) -- (.3,1.7);
	\draw[fill] (0,0) circle (1.5pt);
	\draw[fill] (-1,1) circle (1.5pt);
	\draw[fill] (1,1) circle (1.5pt);
	\draw[fill] (0,1) circle (1.5pt);
	\draw (0,-.5) node[below]{$T$};
	\end{scope}
	\end{tikzpicture}
	\]
\end{definition}

\begin{definition}
	We will say that a functor $\mathcal{K}: \Omega_p \to \mathrm{Top}$ is \emph{equipped with retractions} if for all morphisms $\partial_e: T/e \to T$ in $\Omega_p$ consisting of blowing up a vertex to add a trunk, there is a retraction $r_T: \mathcal{K}(T) \to \mathcal{K}(T/e)$ of $\mathcal{K}(\partial_e)$. Moreover, the retractions are natural, that is, for all $h: S \to T$ in $\Omega_p$, the following square commutes:
	\[
	\xymatrix{
		\mathcal{K}(S) \ar[r]^-{r_S} \ar[d]_-{\mathcal{K}(h)} & \mathcal{K}(S/e) \ar[d]^-{\mathcal{K}(h/e)} \\
		\mathcal{K}(T) \ar[r]_-{r_T} & \mathcal{K}(T/e)
	}
	\]
\end{definition}

\subsection{Computation of homotopy limit}

For $n \geq 0$, we will write $[n]$ for the set $\{0,\ldots,n\}$ and $\mathbf{P}[n]$ for the category of subsets of $[n]$ and inclusions.

\begin{definition}
	For $n \geq 0$, let $\mathbf{C}[n]$ be the full subcategory of
	\[
		\{A \xleftarrow{\sigma} B \xrightarrow{\tau} C \xleftarrow{\upsilon} D\} \times \mathbf{P}[n]
	\]
	of pairs $(l,S)$ where
	\begin{itemize}
		\item $S = [i]$ for some $i \in [n]$ if $l=A$,
		\item $S$ is non-empty if $l \in \{B,C\}$,
		\item $S$ is empty if $l = D$.
	\end{itemize}
	For example, here is a picture of $\mathbf{C}[1]$:
	
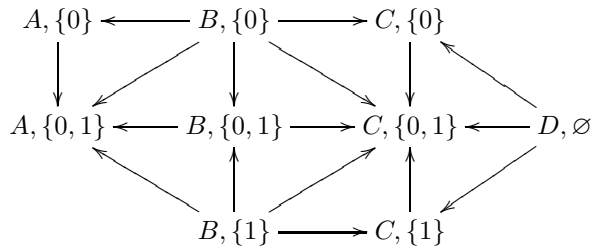
\[
		\xymatrix{
			A,\{0\} \ar[d] & B,\{0\} \ar[l] \ar[ld] \ar[d] \ar[rd] \ar[r] & C,\{0\} \ar[d] \\
			A,\{0,1\} & B,\{0,1\} \ar[l] \ar[r] & C,\{0,1\} & D,\varnothing \ar[l] \ar[lu] \ar[ld] \\
			& B,\{1\} \ar[lu] \ar[u] \ar[ru] \ar[r] & C,\{1\} \ar[u]
		}
	\]
\end{definition}

Recall that for $n \geq 0$, the \emph{topological $n$-simplex} is the topological space:
\[
	\Delta^n = \left\{(x_0,\dots,x_n) \in \mathbb{R}^{n+1} \middle| \text{$\sum_{i=0}^n x_i = 1$ and $x_i \geq 0$ for $0 \leq i \leq n$} \right\}
\]

\begin{lemma}\label{lemmacanonicaliso}
	For $n \geq 0$, there is a canonical isomorphism between the realisation of the nerve of $\mathbf{C}[n]$ and $\Delta^{n+1}$.
\end{lemma}

\begin{proof}
	We describe a map $f: ob(\mathbf{C}[n]) \to \Delta^{n+1}$, where $ob(\mathbf{C}[n])$ is the set of objects of $\mathbf{C}[n]$. Let $(e_0,\ldots,e_{n+1})$ be the standard basis of $\mathbb{R}^{n+2}$. For $S \subset [n]$, let $\max(S)$ be the maximum and $|S|$ be the number of elements of $S$. We define
	\[
		f(l,S) = \begin{cases}
		e_{\max(S)} &\text{if $l=A$,} \\
		\frac{2}{3 |S|} \sum_{i \in S} e_i + \frac{1}{3} e_{n+1} &\text{if $l=B$,} \\
		\frac{1}{3 |S|} \sum_{i \in S} e_i + \frac{2}{3} e_{n+1} &\text{if $l=C$,} \\
		e_{n+1} &\text{if $l=D$.}
		\end{cases}
	\]
	It is easy to check that this map induces the desired isomorphism.
\end{proof}

\begin{definition}
	Let $\Omega_p^*$ be the category of elements of the presheaf $\alpha: \Omega_p \to \mathrm{Set}$ which sends a tree to its set of strata. So $\Omega_p^*$ is the category of trees with a chosen stratum.
\end{definition}

In the following lemma, we will consider $[n]$ as a category, where there is a morphism $i \to j$ if $i < j$. It is a subcategory of $\mathbf{C}[n]$ through the inclusion $i \mapsto (A,[i])$. Also, we will write $\gamma_0 \in \Omega_p^*$ for the trunk, with the unique choice of stratum.

\begin{lemma}\label{lemmatechnical}
	For $n \geq 0$, any functor $T: [n] \to \Omega_p^*$ can be naturally extended to a functor $\bar{T}: \mathbf{C}[n] \to \Omega_p^*$ such that, for $S \subset [n]$ non-empty, the map $\bar{T}(\tau,id_S): \bar{T}(B,S) \to \bar{T}(C,S)$ consists of blowing up a vertex to add a trunk and $\bar{T}(D,\varnothing)=\gamma_0$.
\end{lemma}

\begin{proof}
	Let $T: [n] \to \Omega_p^*$ be a functor. We will extend it to a functor $\bar{T}: \mathbf{C}[n] \to \Omega_p^*$. By assumption, we should have $\bar{T}(A,[i]) = T(i)$ for $i \in [n]$, and $\bar{T}(D,\varnothing)$ should be the trunk. It remains to define $\bar{T}(B,S)$ and $\bar{T}(C,S)$ for $S \subset [n]$ non-empty. For a non-trivial map $i \to j$ in $[n]$, using Remark \ref{remarkinertactiveomegap}, there is $T_{ij} \in \Omega_p^*$ such that $T(i) \to T(j)$ factorises as $T(i) \twoheadrightarrow T_{ij} \rightarrowtail T(j)$, where the first map is active and the second map is inert. Observe that we can add a circle $c_{ij}$ on the tree $T(j)$ such that the tree inside this circle is $T_{ij}$:
	
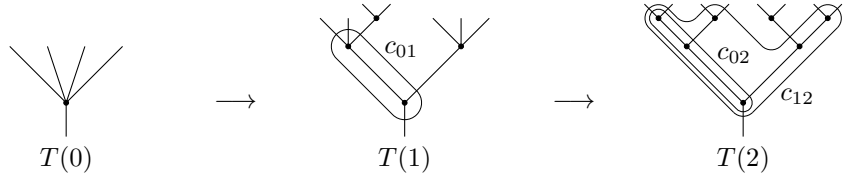
\[
	\begin{tikzpicture}[scale=1.5]
	\draw (0,-.3) -- (0,0) -- (-.5,.5);
	\draw (-.166,.5) -- (0,0) -- (.166,.5);
	\draw (.5,.5) -- (0,0);
	
	\draw (0,-.3) node[below]{$T(0)$};
	
	\draw[fill] (0,0) circle (.6pt);
	
	\draw (1.5,0) node{$\longrightarrow$};
	
	\begin{scope}[shift={(3,0)}]
	\draw (0,-.3) -- (0,0) -- (-.75,.75);
	\draw (-.5,.75) -- (-.5,.5) -- (-.125,.875);
	\draw (-.25,.75) -- (-.375,.875);
	\draw (0,0) -- (.75,.75);
	\draw (.25,.75) -- (.5,.5) -- (.5,.75);
	
	\draw (-.03,.47) node{$\tiny{c_{01}}$};
	
	\draw (0,-.3) node[below]{$T(1)$};
	
	\draw[fill] (0,0) circle (.6pt);
	\draw[fill] (-.5,.5) circle (.6pt);
	\draw[fill] (-.25,.75) circle (.6pt);
	\draw[fill] (.5,.5) circle (.6pt);
	
	\draw[ultra thin] ({-.15*cos(45)},{-.15*sin(45)}) arc (-135:45:.15) -- ({.15*cos(45)-.5},{.15*sin(45)+.5}) arc (45:225:.15) -- ({-.15*cos(45)},{-.15*sin(45)});
	\end{scope}
	
	\draw (4.5,0) node{$\longrightarrow$};
	
	\begin{scope}[shift={(6,0)}]
	\draw (0,-.3) -- (0,0) -- (-.875,.875);
	\draw (-.5,.5) -- (-.125,.875);
	\draw (-.75,.75) -- (-.625,.875);
	\draw (-.25,.75) -- (-.375,.875);
	\draw (0,0) -- (.875,.875);
	\draw (.5,.5) -- (.125,.875);
	\draw (.25,.75) -- (.375,.875);
	\draw (.75,.75) -- (.625,.875);
	
	\draw (.48,.06) node{$\tiny{c_{12}}$};
	\draw (-.08,.42) node{$\tiny{c_{02}}$};
	
	\draw (0,-.3) node[below]{$T(2)$};
	
	\draw[fill] (0,0) circle (.6pt);
	\draw[fill] (-.5,.5) circle (.6pt);
	\draw[fill] (-.75,.75) circle (.6pt);
	\draw[fill] (-.25,.75) circle (.6pt);
	\draw[fill] (.5,.5) circle (.6pt);
	\draw[fill] (.25,.75) circle (.6pt);
	\draw[fill] (.75,.75) circle (.6pt);
	
	\draw[ultra thin] ({-.08*cos(45)},{-.08*sin(45)}) arc (-135:45:.08) -- ({.08*cos(45)-.75},{.08*sin(45)+.75});
	\draw[ultra thin] ({-.08*cos(45)},{-.08*sin(45)}) -- ({-.08*cos(45)-.75},{-.08*sin(45)+.75});
	\draw[ultra thin] ({.08*cos(45)-.75},{.08*sin(45)+.75}) arc (45:225:.08);
	
	\draw[ultra thin] ({-.12*cos(45)},{-.12*sin(45)}) arc (-135:-45:.12);
	\draw[ultra thin] ({.12*cos(45)},{-.12*sin(45)}) -- ({.12*cos(45)+.75},{-.12*sin(45)+.75});
	\draw[ultra thin] ({.12*cos(45)+.75},{-.12*sin(45)+.75}) arc (-45:135:.12);
	\draw[ultra thin] ({-.12*cos(45)+.75},{.12*sin(45)+.75}) -- ({.12*cos(45)+.25},{.36*sin(45)+.25});
	\draw[ultra thin] ({-.12*cos(45)+.25},{.36*sin(45)+.25}) arc (-135:-45:.12);
	\draw[ultra thin] ({-.12*cos(45)+.25},{.36*sin(45)+.25}) -- ({.12*cos(45)-.25},{.12*sin(45)+.75});
	\draw[ultra thin] ({.12*cos(45)-.25},{.12*sin(45)+.75}) arc (45:135:.12);
	\draw[ultra thin] ({-.12*cos(45)-.25},{.12*sin(45)+.75}) -- ({.12*cos(45)-.5},{.36*sin(45)+.5});
	\draw[ultra thin] ({-.12*cos(45)-.5},{.36*sin(45)+.5}) arc (-135:-45:.12);
	\draw[ultra thin] ({-.12*cos(45)-.5},{.36*sin(45)+.5}) -- ({.12*cos(45)-.75},{.12*sin(45)+.75});
	
	\draw[ultra thin] ({.12*cos(45)-.75},{.12*sin(45)+.75}) arc (45:225:.12);

	\draw[ultra thin] ({-.12*cos(45)},{-.12*sin(45)}) -- ({-.12*cos(45)-.75},{-.12*sin(45)+.75});
	\end{scope}
	\end{tikzpicture}
	\]
	We define $\bar{T}(B,S)$ as the tree obtained from $T(\max(S))$ by contracting all the internal edges except the ones that are crossed by a circle $c_{ij}$ for $i,j \in S$ and $j=\max(S)$. Note that in particular, $\bar{T}(B,\{i\})$ is the tree obtained from $T(i)$ by contracting all the internal edges. In the case of the free living edge, we add a unary vertex. For example, if we start with $T: [2] \to \Omega_p^*$ as in the previous picture, we will get the following trees $\bar{T}(B,S)$ for $S \subset [2]$ non-empty (forgetting about the chosen strata):
	
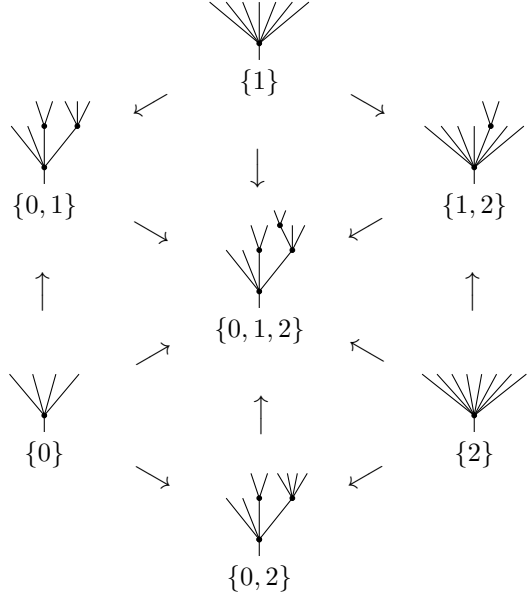
\[
		\begin{tikzpicture}[scale=1.1]
			\begin{scope}
			\draw (0,0) -- (0,-.2);
			\draw (-.4,.5) -- (0,0) -- (.4,.5);
			\draw (-.2,.5) -- (0,0) -- (0,.5);
			\draw (.25,.8) -- (.4,.5) -- (.55,.8);
			\draw (.4,.5) -- (.4,.8);
			\draw (-.1,.8) -- (0,.5) -- (.1,.8);
			\draw (.18,.98) -- (.25,.8) -- (.32,.98);
			\draw[fill] (0,0) circle (.8pt);
			\draw[fill] (0,.5) circle (.8pt);
			\draw[fill] (.4,.5) circle (.8pt);
			\draw[fill] (.25,.8) circle (.8pt);
			\draw (0,-.2) node[below]{$\tiny{\{0,1,2\}}$};
			\end{scope}
			\begin{scope}[shift={({3*cos(30)},{3*sin(30)})}]
			\draw (0,-.2) -- (0,0) -- (0,.5);
			\draw (-.6,.5) -- (0,0) -- (.6,.5);
			\draw (-.4,.5) -- (0,0) -- (.4,.5);
			\draw (-.2,.5) -- (0,0) -- (.2,.5);
			\draw (.1,.8) -- (.2,.5) -- (.3,.8);
			\draw[fill] (.2,.5) circle (.8pt);
			\draw[fill] (0,0) circle (.8pt);
			\draw (0,-.2) node[below]{$\tiny{\{1,2\}}$};
			\end{scope}
			\begin{scope}[shift={({3*cos(90)},{3*sin(90)})}]
			\draw (0,-.2) -- (0,0) -- (0,.5);
			\draw (-.6,.5) -- (0,0) -- (.6,.5);
			\draw (-.4,.5) -- (0,0) -- (.4,.5);
			\draw (-.2,.5) -- (0,0) -- (.2,.5);
			\draw[fill] (0,0) circle (.8pt);
			\draw (0,-.2) node[below]{$\tiny{\{1\}}$};
			\end{scope}
			\begin{scope}[shift={({3*cos(150)},{3*sin(150)})}]
			\draw (0,0) -- (0,-.2);
			\draw (-.4,.5) -- (0,0) -- (.4,.5);
			\draw (-.2,.5) -- (0,0) -- (0,.5);
			\draw (.25,.8) -- (.4,.5) -- (.55,.8);
			\draw (.4,.5) -- (.4,.8);
			\draw (-.1,.8) -- (0,.5) -- (.1,.8);
			\draw[fill] (0,0) circle (.8pt);
			\draw[fill] (0,.5) circle (.8pt);
			\draw[fill] (.4,.5) circle (.8pt);
			\draw (0,-.2) node[below]{$\tiny{\{0,1\}}$};
			\end{scope}
			\begin{scope}[shift={({3*cos(210)},{3*sin(210)})}]
			\draw (0,0) -- (0,-.2);
			\draw (-.42,.5) -- (0,0) -- (.42,.5);
			\draw (-.14,.5) -- (0,0) -- (.14,.5);
			\draw[fill] (0,0) circle (.8pt);
			\draw (0,-.2) node[below]{$\tiny{\{0\}}$};
			\end{scope}
			\begin{scope}[shift={({3*cos(270)},{3*sin(270)})}]
			\draw (0,0) -- (0,-.2);
			\draw (-.4,.5) -- (0,0) -- (.4,.5);
			\draw (-.2,.5) -- (0,0) -- (0,.5);
			\draw (.22,.8) -- (.4,.5) -- (.58,.8);
			\draw (.34,.8) -- (.4,.5) -- (.46,.8);
			\draw (-.1,.8) -- (0,.5) -- (.1,.8);
			\draw[fill] (0,0) circle (.8pt);
			\draw[fill] (0,.5) circle (.8pt);
			\draw[fill] (.4,.5) circle (.8pt);
			\draw[fill] (0,0) circle (.8pt);
			\draw (0,-.2) node[below]{$\tiny{\{0,2\}}$};
			\end{scope}
			\begin{scope}[shift={({3*cos(330)},{3*sin(330)})}]
			\draw (0,0) -- (0,-.2);
			\draw (-.63,.5) -- (0,0) -- (.63,.5);
			\draw (-.45,.5) -- (0,0) -- (.45,.5);
			\draw (-.27,.5) -- (0,0) -- (.27,.5);
			\draw (-.09,.5) -- (0,0) -- (.09,.5);
			\draw[fill] (0,0) circle (.8pt);
			\draw (0,-.2) node[below]{$\tiny{\{2\}}$};
			\end{scope}
			\draw ({1.5*cos(30)},{1.5*sin(30)}) node[rotate=30]{$\longleftarrow$};
			\draw ({1.5*cos(90)},{1.5*sin(90)}) node[rotate=90]{$\longleftarrow$};
			\draw ({1.5*cos(150)},{1.5*sin(150)}) node[rotate=150]{$\longleftarrow$};
			\draw ({1.5*cos(210)},{1.5*sin(210)}) node[rotate=210]{$\longleftarrow$};
			\draw ({1.5*cos(270)},{1.5*sin(270)}) node[rotate=270]{$\longleftarrow$};
			\draw ({1.5*cos(330)},{1.5*sin(330)}) node[rotate=330]{$\longleftarrow$};
			\draw ({3*cos(30)},0) node[rotate=90]{$\longrightarrow$};
			\draw ({1.5*cos(30)},2.25) node[rotate=-30]{$\longrightarrow$};
			\draw ({-1.5*cos(30)},2.25) node[rotate=30]{$\longleftarrow$};
			\draw ({-3*cos(30)},0) node[rotate=90]{$\longrightarrow$};
			\draw ({-1.5*cos(30)},-2.25) node[rotate=-30]{$\longrightarrow$};
			\draw ({1.5*cos(30)},-2.25) node[rotate=30]{$\longleftarrow$};
		\end{tikzpicture}
	\]
	$\bar{T}(C,S)$ is the same tree as $\bar{T}(B,S)$ but where the vertex in the most inside circle is blown up to add a trunk in the chosen strata. Such $\bar{T}$ defined on objects extends to a functor. For $i \in [n]$, $\bar{T}(\sigma,id_{[i]})$ is an active. For $S \subset [n]$ non-empty, $\bar{T}(\tau,id_S)$ consists of blowing up a vertex to add a trunk, as required. For $i \notin S$, the maps $\bar{T}(id_B, S \to S \amalg i)$ and $\bar{T}(id_C, S \to S \amalg i)$ are active if $i \leq \max(S)$ and inert if $i > \max(S)$. Finally, the map $\bar{T}(\upsilon, \varnothing \to S)$ is the inert map given by inclusion of $\gamma_0$ to the trunk of $\bar{T}(C,S)$ which was added from $\bar{T}(B,S)$.
\end{proof}

Let $\Delta^\bullet: \Delta \to \mathrm{Top}$ be the cosimplicial space which sends $n$ to $\Delta^n$.

\begin{lemma}\label{lemmadescriptionholim}
	Let $\mathcal{C}$ be a category and $\mathcal{F}: \mathcal{C} \to \mathrm{Top}$ be a functor. Then giving $\beta \in \holim \mathcal{F}$ is the same as naturally describing, for all $n \geq 0$ and $T: [n] \to \mathcal{C}$, a map $\beta(T): \Delta^n \to \mathcal{F}T(n)$. \emph{Naturally} means that for all $f: [m] \to [n]$,
	\begin{equation}\label{equationnaturalholim}
		\beta(T) \Delta^f = \mathcal{F}T(f(m) \to n) \beta(Tf).
	\end{equation}
\end{lemma}

\begin{proof}
	Let $c \mathcal{F}: \Delta \to \mathrm{Top}$ be the cosimplicial space given by
	\[
		(c \mathcal{F})_n = \coprod_{T: [n] \to C} \mathcal{F}T(n).
	\]
	The homotopy limit of $\mathcal{F}$ is the totalization of $c \mathcal{F}$. This gives us the desired result.
\end{proof}

Let $\pi: \Omega_p^* \to \Omega_p$ be the projection functor to the base. Concretely, $\pi$ forgets the strata.

\begin{lemma}\label{lemmaholimcontractible}
	Let $\mathcal{K}: \Omega_p \to \mathrm{Top}$ equipped with retractions. Then the homotopy limit of $\pi^*\mathcal{K}$ is given by its value at the trunk.
\end{lemma}

\begin{proof}
	Let us write $\mathcal{L} = \pi^*(\mathcal{K})$ and $\gamma_0 \in \Omega_p^*$ be the trunk. We will construct two maps
	\[
	\xymatrix{
		\holim \mathcal{L} \ar@<.5ex>[r]^-q  & \mathcal{L}(\gamma_0) \ar@<.5ex>[l]^-j
	}
	\]
	and prove that they are homotopy inverse of each other. 
	The map $q$ is the projection which sends $\beta \in \holim \mathcal{L}$ to the point given by the map $\beta(\gamma_0): \Delta^0 \to \mathcal{L}(\gamma_0)$ of Lemma \ref{lemmadescriptionholim}. We can construct the map $j$ using Lemma \ref{lemmatechnical} and the assumption that $\mathcal{K}$ is equipped with retractions. Let $\xi \in \mathcal{L}(\gamma_0)$. A functor $T: [n] \to \Omega_p^*$ can be extended to $\bar{T}: \mathbf{C}[n] \to \Omega_p^*$ according to Lemma \ref{lemmatechnical}. We have the zigzag
	\begin{equation}\label{zigzag}
		\xymatrix{
			\gamma_0 \ar[rr]^-{\bar{T}(\nu,id_{[n]})} && \bar{T}(C,[n]) && \bar{T}(B,[n]) \ar[ll]_-{\bar{T}(\tau,id_{[n]})}  \ar[rr]^-{\bar{T}(\sigma,id_{[n]})} && T(n)
		}
	\end{equation}
	where the map in the middle consists of blowing up a vertex to add the trunk. Applying $\mathcal{L}$ to this zigzag gives us another zigzag, where the map in the middle has a retraction, since $\mathcal{K}$ is equipped with retractions. Then $j(\xi)(T)$ given is by the composite $\Delta^n \to \mathcal{L}(\gamma_0) \to \mathcal{L}T(n)$, where the first map is constant to $\xi$ and the second map is the composite obtained by applying $\mathcal{L}$ to \ref{zigzag} and replacing the map in the middle by its retraction. 
	The composite $qj$ is the identity. We will now prove that the composite $jq$ is also homotopic to the identity.
	
	Let us fix $\beta \in \holim \mathcal{L}$ and a functor $T: [n] \to \Omega_p^*$, which can again be extended to $\bar{T}: \mathbf{C}[n] \to \Omega_p^*$. We will describe a map $H(\beta)(T): \Delta^{n+1} \to \mathcal{L}T(n)$. Let $i: [n+1] \to \mathbf{C}[n]$ non-degenerate, that is $i(f) \neq id$ if $f \neq id$. We define $H(\beta)(T)|_i: \Delta^{n+1} \to \mathcal{L}T(n)$ as follows. Using Lemma \ref{lemmadescriptionholim}, we can associate to the functor $\bar{T}i$ a map $\beta(\bar{T}i) : \Delta^{n+1} \to \mathcal{L}\bar{T}i(n+1)$. Note that since $i$ is non-degenerate, $i(n+1)=(A,[n])$ or $i(n+1)=(C,[n])$. So $\bar{T}i(n+1) = T(n)$ or $\bar{T}i(n+1)=\bar{T}(C,[n])$. We define
	\[
		H(\beta)(T)|_i = \begin{cases}
		\beta(\bar{T}i) &\text{if $i(n+1) = (A,[n])$,}\\
		\mathcal{L}\bar{T}(\sigma,id_{[n]}) \mathcal{L}\bar{T}(\tau,id_{[n]})^{-1} \beta(\bar{T}i) &\text{if $i(n+1)=(C,[n])$,}
		\end{cases}
	\]
	where $\mathcal{L}\bar{T}(\tau,id_{[n]})^{-1}$ is the retraction of $\mathcal{L}\bar{T}(\tau,id_{[n]})$. According to Lemma \ref{lemmacanonicaliso}, $i$ induces a canonical map $|N|(i): \Delta^{n+1} \to \Delta^{n+1}$. We can define the restriction of $H(\beta)(T)$ to the image of $|N|(i)$ to be given by $H(\beta)(T)|_i$. It remains to check that $H(\beta)(T)$ is well-defined. Let $i_A,i_C: [n+1] \to \mathbf{C}[n]$ be two non-degenerate functors such that $i_A(n+1) = (A,[n])$, $i_C(n+1) = (C,[n])$ and $i := i_A d_{n+1} = i_C d_{n+1}$, where $d_{n+1}: [n] \to [n+1]$ is the inclusion. Note that $i_A(n \to n+1) = (\sigma,id_{[n]})$ and $i_C(n \to n+1) = (\tau,id_{[n]})$. 
	Using the naturality \ref{equationnaturalholim}, we have
	\begin{equation*}
	\begin{split}
	H(\beta)(T)|_{i_A} \delta_{n+1} &= \beta(\bar{T} i_A) \delta_{n+1}, \\
	&= \mathcal{L}\bar{T}(\sigma,id_{[n]}) \beta(\bar{T} i), \\
	&= \mathcal{L}\bar{T}(\sigma,id_{[n]}) \mathcal{L}\bar{T}(\tau,id_{[n]})^{-1} \mathcal{L}\bar{T}(\tau,id_{[n]}) \beta(\bar{T} i), \\
	&= \mathcal{L}\bar{T}(\sigma,id_{[n]}) \mathcal{L}\bar{T}(\tau,id_{[n]})^{-1} \beta(\bar{T} i_C) \delta_{n+1}, \\
	&= H(\beta)(T)|_{i_C} \delta_{n+1},
	\end{split}
	\end{equation*}
	where $\delta_{n+1} = \Delta^{d_{n+1}}: \Delta^n \to \Delta^{n+1}$. 
	This proves that $H(\beta)(T)$ is well-defined.
	
	Finally, the desired homotopy $H: \left[0,1\right] \times \holim \mathcal{L} \to \holim \mathcal{L}$ is given by $H(t,\beta)(T)(x) = H(\beta)(T)((1-t)x,t)$.
\end{proof}

\subsection{The case $(m,n)=(1,3)$}

\begin{theorem}
	We do have the fibration sequence \ref{fibrationsequence} for $(m,n) = (1,3)$ with the extra condition that $X$ is such that $v^*X$ is equipped with retractions.
\end{theorem}

\begin{proof}
	Since $\pi: \Omega_p^* \to \Omega_p$ is a discrete fibration, the left Kan extension can be computed as a coproduct over fibres. In case of the terminal presheaf, we get the coproduct of $1$ over the fibres of $\pi$, that is the fibres of $\pi$ themselves. So, $\pi_!(1) = \alpha$ and by adjunction,
	\[
	\mathrm{Map}_{[\Omega_p,\mathrm{SSet}]} (\alpha,Y) \sim \mathrm{Map}_{[\Omega_p^*,\mathrm{SSet}]} (1,\pi^* Y) = \holim_{\Omega_p^*} \pi^* Y,
	\]
	where $Y := v^*(X)$. We get the desired result by combining Lemma \ref{lemmabimodulesbimodules}, Lemma \ref{lemmafibrationsequence} and Lemma \ref{lemmaholimcontractible}.
\end{proof}

\begin{definition}
	We will call \emph{hyperoperads} algebras of the polynomial monad $\mathbf{Id}^{+3}$. We will write $\mathrm{HOp}$ for the category of hyperoperads.
\end{definition}

\begin{corollary}\label{corollarytripledelooping}
	Let $\mathcal{O}$ be a multiplicative hyperoperad. Assume that for all planar trees $T$ with zero or one vertices, that is, the free living edge or the corollas, $\mathcal{O}_T$ is contractible. Assume further that $\mathcal{O}^\bullet$ is equipped with retractions. Then we have a weak equivalence
	\[
	\Omega^3 \mathrm{Map}_{\mathrm{HOp}} (\zeta,u^*(\mathcal{O})) \sim \holim_{\Omega_p} \mathcal{O}^\bullet.
	\]
\end{corollary}

\subsection{Example: desymmetrisation of the Kontsevich operad}

Recall \cite[Section 3.5]{bwdquasitame} that for any polynomial monad $T$, there is a canonical map of polynomial monads from $T^+$ to the polynomial monad for symmetric operads. In particular, for $T=\mathrm{NOp}$, this map of polynomial monad induces a \emph{desymmetrisation} functor $des: \mathrm{SOp} \to \mathrm{HOp}$, where $\mathrm{SOp}$ is the category of symmetric operads. Explicitly, it is given, for a planar tree $T$, by $des(\mathcal{P})(T) = \mathcal{P}(|T|)$, where $|T|$ is the set of vertices of $T$.

Let $\mathcal{O}$ be a multiplicative hyperoperad. If $u^*(\mathcal{O})$ is the desymmetrisation of a reduced symmetric operad, then $\mathcal{O}^\bullet$ is equipped with retractions. Indeed, let $\mathcal{P}$ be the symmetric operad such that $des(\mathcal{P}) = u^*(\mathcal{O})$. The retraction, for a morphism $\partial_e: T /e \to T$ in $\Omega_p$, is given by the map
\[
	\mathcal{O}(T) = \mathcal{P}(|T|) \otimes \mathcal{P}(\varnothing) \xrightarrow{\circ_v} \mathcal{P}(|T| \setminus \{v\}) = \mathcal{O}(T/e),
\]
where $v \in |T|$ is the vertex directly above $e$, and $\circ_v$ is the multiplication of the symmetric operad given in terms of partial operations \cite{markl2008operads}. We will now give a non-trivial example of multiplicative hyperoperad.

For $m \geq 2$, let $\tilde{C}_m(n)$ be the quotient of the configuration space
\[
\mathrm{Conf}_n(\mathbb{R}^m) = \{ (x_1,\ldots,x_n) \in (\mathbb{R}^m)^n, \text{ $x_i \neq x_j$ if $i \neq j$} \}
\]
with respect to the action of the group $G_m = \{ x \mapsto \lambda x + v | \lambda > 0, v \in \mathbb{R}^m \}$.

\begin{definition} \cite[Definition 12]{kontsevich}
	The \emph{Kontsevich operad} $\mathcal{K}_m(n)$ is the closure of the image $\tilde{C}_m(n)$ in $(S^{m-1})^{n \choose 2}$ under the map
	\[
	G_m \cdot (x_1,\ldots,x_n) \mapsto \left( \frac{x_j - x_i}{|x_j - x_i|} \right)_{1 \leq i < j \leq n}
	\]
	Set-theoretically, the operad $\mathcal{K}_m$ is the same as the free operad generated by the symmetric collection of sets $(\tilde{C}_m(n))_{n \geq 0}$.
\end{definition}

\begin{lemma}
	The hyperoperad obtained by desymmetrisation of the Kontsevich operad $\mathcal{K}_m$ has a multiplicative structure.
\end{lemma}

\begin{proof}
	We write $x = (x_{ij})_{i \neq j \in |T|}$ for an element of $\mathcal{K}_m(T)$. 
	Let $e_1,e_2 \in S^{m-1}$ given by the standard inclusion of $\mathbb{R}^2$ into $\mathbb{R}^m$. Let $x(T) \in \mathcal{K}_m(T)$ given by
	\[
	x(T)_{ij} = \begin{cases}
	e_1 &\text{if $i$ is below $j$,} \\
	e_2 &\text{if $i$ is to the left of $j$.} \end{cases}
	\]
	It is easy to check that this does give a multiplicative structure.
\end{proof}

In particular we can apply Corollary \ref{corollarytripledelooping} to the desymmetrisation of the Kontsevich operad.

\bibliographystyle{plain}
\bibliography{references}

\begin{thebibliography}{10}

\bibitem{baezdolan}
John~C Baez and James Dolan.
\newblock Higher-dimensional algebra iii. n-categories and the algebra of
  opetopes.
\newblock {\em Advances in Mathematics}, 135(2):145--206, 1998.

\bibitem{batanin}
Michael Batanin.
\newblock The {E}ckmann--{H}ilton argument and higher operads.
\newblock {\em Advances in Mathematics}, 217(1):334--385, 2008.

\bibitem{bataninberger}
Michael Batanin and Clemens Berger.
\newblock Homotopy theory for algebras over polynomial monads.
\newblock {\em Theory and Applications of Categories}, 32(6):148--253, 2017.

\bibitem{batanindeleger}
Michael Batanin and Florian De~Leger.
\newblock Polynomial monads and delooping of mapping spaces.
\newblock {\em Journal of Noncommutative Geometry}, 13(4):1521+, 2019.

\bibitem{bwdgrothendieck}
Michael Batanin, David White, and Florian De~Leger.
\newblock Model structures on operads and algebras from a global perspective.
\newblock In preparation.

\bibitem{bwdquasitame}
Michael Batanin, David White, and Florian De~Leger.
\newblock Quasi-tame substitudes and the grothendieck construction.
\newblock In preparation.

\bibitem{berger}
Clemens Berger.
\newblock Moment categories and operads.
\newblock {\em Theory and Applications of Categories}, 38(39):1485--1537, 2022.

\bibitem{cisinski}
Denis-Charles Cisinski.
\newblock {\em Les pr{\'e}faisceaux comme mod{\`e}les des types d'homotopie}.
\newblock Soci{\'e}t{\'e} math{\'e}matique de France, 2006.

\bibitem{deleger}
Florian De~Leger.
\newblock Cofinal morphism of polynomial monads and double delooping.
\newblock {\em arXiv preprint arXiv:2205.09149}, 2022.

\bibitem{ducoulombier}
Julien Ducoulombier.
\newblock Delooping derived mapping spaces of bimodules over an operad.
\newblock {\em Journal of Homotopy and Related Structures}, 14(2):411--453,
  2019.

\bibitem{ducoulombierturchin}
Julien Ducoulombier and Victor Turchin.
\newblock Delooping the functor calculus tower.
\newblock {\em Proceedings of the London Mathematical Society},
  124(6):772--853, 2022.

\bibitem{dwyerhess}
William Dwyer and Kathryn Hess.
\newblock Long knots and maps between operads.
\newblock {\em Geometry \& Topology}, 16(2):919--955, 2012.

\bibitem{gambinokock}
Nicola Gambino and Joachim Kock.
\newblock Polynomial functors and polynomial monads.
\newblock In {\em Mathematical proceedings of the cambridge philosophical
  society}, volume 154, pages 153--192. Cambridge University Press, 2013.

\bibitem{kock}
Joachim Kock.
\newblock Polynomial functors and trees.
\newblock {\em International Mathematics Research Notices}, 2011(3):609--673,
  2011.

\bibitem{KJBM}
Joachim Kock, Andr{\'e} Joyal, Michael Batanin, and Jean-Fran{\c{c}}ois
  Mascari.
\newblock Polynomial functors and opetopes.
\newblock {\em Advances in Mathematics}, 224(6):2690--2737, 2010.

\bibitem{kontsevich}
Maxim Kontsevich.
\newblock Operads and motives in deformation quantization.
\newblock {\em Letters in Mathematical Physics}, 48(1):35--72, 1999.

\bibitem{markl2008operads}
Martin Markl.
\newblock Operads and props.
\newblock {\em Handbook of algebra}, 5:87--140, 2008.

\bibitem{moerdijktoen}
Ieke Moerdijk and Bertrand To{\"e}n.
\newblock {\em Simplicial methods for operads and algebraic geometry}.
\newblock 2010.

\bibitem{rezk}
Charles Rezk.
\newblock Every homotopy theory of simplicial algebras admits a proper model.
\newblock {\em Topology and its Applications}, 119(1):65--94, 2002.

\bibitem{sinha1}
Dev Sinha.
\newblock Operads and knot spaces.
\newblock {\em Journal of the American Mathematical Society}, 19(2):461--486,
  2006.

\bibitem{turchincosimplicial}
Victor Turchin.
\newblock Hodge-type decomposition in the homology of long knots.
\newblock {\em Journal of Topology}, 3(3):487--534, 2010.

\bibitem{turchin}
Victor Turchin.
\newblock Delooping totalization of a multiplicative operad.
\newblock {\em Journal of Homotopy and Related Structures}, 9(2):349--418,
  2014.

\end{thebibliography}

\end{document}